\documentclass[final,3p,times]{elsarticle}

\usepackage{amssymb}
\usepackage{amsmath}
\usepackage{amsthm}

\journal{Applied and Computational Harmonic Analysis}

\newcommand{\bbS}{\mathbb{S}}
\newcommand{\bbR}{\mathbb{R}}

\newcommand{\calK}{\mathcal{K}}
\newcommand{\calO}{\mathcal{O}}
\newcommand{\calU}{\mathcal{U}}

\newcommand{\rmd}{\mathrm{d}}
\newcommand{\rme}{\mathrm{e}}
\newcommand{\rmI}{\mathrm{I}}
\newcommand{\rmO}{\mathrm{O}}

\newcommand{\Snn}{\mathbb{S}_{\mathrm{nn}}}

\newcommand{\abs}[1]{|{#1}|}

\newcommand{\bigparen}[1]{\bigl({#1}\bigr)}
\newcommand{\Bigparen}[1]{\Bigl({#1}\Bigr)}
\newcommand{\biggparen}[1]{\biggl({#1}\biggr)}

\newcommand{\bigbracket}[1]{\bigl[{#1}\bigr]}
\newcommand{\Bigbracket}[1]{\Bigl[{#1}\Bigr]}

\newcommand{\set}[1]{\{{#1}\}}
\newcommand{\bigset}[1]{\bigl\{{#1}\bigr\}}

\newcommand{\Biggset}[1]{\Biggl\{{#1}\Biggr\}}

\newcommand{\norm}[1]{\|{#1}\|}

\newcommand{\biggnorm}[1]{\biggl\|{#1}\biggr\|}

\newcommand{\ip}[2]{\langle{#1},{#2}\rangle}

\theoremstyle{plain}
\newtheorem{theorem}{Theorem}
\newtheorem{lemma}[theorem]{Lemma}

\theoremstyle{definition}

\begin{document}
\begin{frontmatter}
\title{Group-theoretic constructions of erasure-robust frames}

\author[AFIT]{Matthew Fickus}
\ead{Matthew.Fickus@gmail.com}
\author[MO]{John Jasper}
\author[AFIT]{Dustin G.~Mixon}
\author[MO]{Jesse Peterson}

\address[AFIT]{Department of Mathematics and Statistics, Air Force Institute of Technology, Wright-Patterson Air Force Base, OH 45433, USA}
\address[MO]{Department of Mathematics, University of Missouri, Columbia, MO 65211, USA}

\begin{abstract}
In the field of compressed sensing, a key problem remains open: to explicitly construct matrices with the restricted isometry property (RIP) whose performance rivals those generated using random matrix theory.  In short, RIP involves estimating the singular values of a combinatorially large number of submatrices, seemingly requiring an enormous amount of computation in even low-dimensional examples.  In this paper, we consider a similar problem involving submatrix singular value estimation, namely the problem of explicitly constructing numerically erasure robust frames (NERFs).  Such frames are the latest invention in a long line of research concerning the design of linear encoders that are robust against data loss.  We begin by focusing on a subtle difference between the definition of a NERF and that of an RIP matrix, one that allows us to introduce a new computational trick for quickly estimating NERF bounds.  In short, we estimate these bounds by evaluating the frame analysis operator at every point of an $\varepsilon$-net for the unit sphere.  We then borrow ideas from the theory of group frames to construct explicit frames and $\varepsilon$-nets with such high degrees of symmetry that the requisite number of operator evaluations is greatly reduced.  We conclude with numerical results, using these new ideas to quickly produce decent estimates of NERF bounds which would otherwise take an eternity.  Though the more important RIP problem remains open, this work nevertheless demonstrates the feasibility of exploiting symmetry to greatly reduce the computational burden of similar combinatorial linear algebra problems.
\end{abstract}

\end{frontmatter}


\section{Introduction}
Throughout this work, let $\Phi$ denote a short, real $M\times N$ matrix; though some of the results presented here generalize to the complex setting, others do not.  That is, let $M\leq N$ and let $\Phi=[\varphi_1\ \cdots\ \varphi_N]$ where $\set{\varphi_n}_{n=1}^{N}\subseteq\bbR^M$ are the columns of $\Phi$.  For any  $\calK\subseteq\set{1,\dotsc,N}$, define $K:=\abs{\calK}$ and consider the $M\times K$ submatrix $\Phi_{\calK}$ of $\Phi$ obtained by concatenating the columns $\set{\varphi_n}_{n\in\calK}$.  Fixing $K$, both the restricted isometry property and numerically erasure-robust frames are defined in terms of the extreme singular values of $\Phi_\calK$ as $\calK$ ranges over all $K$-element subsets of $\set{1,\dotsc,N}$; the key difference between the two depends on whether the $\Phi_\calK$'s are tall ($K\leq M$) or short ($M\leq K\leq N$).

In particular, for a fixed $K\leq M$ and $\delta>0$, the matrix $\Phi$ is said to have the \textit{$(K,\delta)$-restricted isometry property} (RIP) if for any $K$-element subset $\calK\subseteq\set{1,\dotsc,N}$, we have
\begin{equation}
\label{equation.definition of RIP}
(1-\delta)\sum_{n\in\calK}\abs{y(n)}^2
\leq\biggnorm{\sum_{n\in\calK}y(n)\varphi_n}^2
\leq(1+\delta)\sum_{n\in\calK}\abs{y(n)}^2,
\quad\forall y\in\bbR^N.
\end{equation}
This means that $\Phi$ acts like a near isometry over all vectors whose support is at most $K$.
More precisely, \eqref{equation.definition of RIP} is equivalent to having $(1-\delta)\leq\norm{\Phi_\calK z}^2\leq(1+\delta)$ for all unit norm $z\in\bbR^K$, which in turn is equivalent to having all of the eigenvalues of the Gram submatrix $\Phi_\calK^*\Phi_{\calK}^{}$ lie in the interval $[1-\delta,1+\delta]$.
This can be viewed as a weakened version of a requirement that $\Phi_\calK^*\Phi_{\calK}^{}=\rmI$.
That is, RIP requires every $K$-element subcollection $\set{\varphi_n}_{n\in\calK}$ of $\set{\varphi_n}_{n=1}^{N}$ to be nearly orthonormal, forming a nice Riesz basis for its span.

RIP matrices are important because they permit efficient sensing and reconstruction of sparse signals---see Theorem~1.3 in~\cite{Candes:08}, for example---a problem which can otherwise be NP-hard~\cite{Natarajan:95}.
In particular, for a fixed $M$, $N$ and $\delta>0$, the goal in such applications is to construct $M\times N$ matrices $\Phi$ which are $(K,\delta)$-RIP where $K$ is as large as possible: larger values of $K$ correspond to a larger class of $N$-dimensional signals that can be fully sensed using only $M$ measurements.
Unfortunately, it has proven notoriously difficult to deterministically construct RIP matrices for large values of $K$.
To be precise, whereas many randomly constructed matrices~\cite{BaraniukDDW:08,RudelsonV:08} have a high probability of being $(K,\delta)$-RIP for any $K\leq CM/\mathrm{polylog}(N)$, all known deterministic methods are either stuck at the ``square-root bottleneck" of having $K\leq C\sqrt{M}$---see~\cite{DeVore:07}, for example---or go only slightly beyond it~\cite{BourgainDFKK:11}.
As detailed in~\cite{BandeiraFMW:arxiv12}, these difficulties stem from the fact that checking that a given $M\times N$ matrix $\Phi$ is $(K,\delta)$-RIP involves estimating the singular values of \smash{$\binom{N}{K}$} submatrices of $\Phi$ of size $M\times K$.
Since singular values are hard to estimate analytically and since \smash{$\binom{N}{K}$} is enormous for even modest choices of $N$ and $K$, this problem poses serious challenges for any analytical or numerical approach.

In this paper, we provide a hybrid analytical/numerical approach for solving a problem that is very similar to that of deterministically constructing RIP matrices.
Though, as we discuss below, our techniques do not immediately generalize to the RIP setting, they nevertheless demonstrate the feasibility of numerically estimating the singular values of a combinatorially large number of matrices.
To be precise, the purpose of this paper is to provide new deterministic constructions of \textit{numerically erasure-robust frames} (NERFs): given $M\leq K\leq N$ and $0<\alpha\leq\beta<\infty$, we say that $\set{\varphi_n}_{n=1}^{N}$ is a $(K,\alpha,\beta)$-NERF for $\bbR^M$ if for any $K$-element subset $\calK\subseteq\set{1,\dotsc,N}$, we have
\begin{equation}
\label{equation.definition of NERF}
\alpha\norm{x}^2
\leq\sum_{n\in\calK}\abs{\ip{x}{\varphi_n}}^2
\leq\beta\norm{x}^2,
\quad\forall x\in\bbR^M.
\end{equation}
This means that $\set{\varphi_n}_{n\in\calK}$ is a frame for $\bbR^M$ with frame bounds $\alpha,\beta$ regardless of the choice of $\calK$.
Note~\eqref{equation.definition of NERF} is equivalent to having $\alpha\leq\norm{\Phi_{\calK}^* x}^2\leq\beta$ for all unit norm $x\in\bbR^M$, which in turn is equivalent to having all the eigenvalues of the subframe operator $\Phi_{\calK}^{}\Phi_{\calK}^*$ lie in the interval $[\alpha,\beta]$.

Thus, when $\alpha\approx\beta$, we see that a NERF is a short matrix for which every \textit{short} submatrix of a given size is well-conditioned; meanwhile an RIP matrix is a short matrix for which every \textit{tall} submatrix of a given size is well-conditioned.
This subtle difference in definition leads to vastly different envisioned applications; while RIP matrices are primarily intended for compressed sensing, NERFs are the latest invention~\cite{FickusM:12} in a long line of research~\cite{GoyalKK:01,CasazzaK:03,HolmesP:laa04,PuschelK:dcc05,AlexeevCM:arxiv11} concerning the design of linear encoders that are robust against data loss.
Here, one encodes $x\in\bbR^M$ as a higher-dimensional vector $\Phi^* x\in\bbR^N$ which is then transmitted in a channel with erasures and additive noise, yielding $y=\Phi_{\calK}^*x+\varepsilon$, where $\calK$ corresponds to the entries of $y$ that were not erased.
The problem is then to reconstruct $x$ from $y$.
In the standard least squares approach, this reconstruction is achieved by solving the normal equations \smash{$\Phi_{\calK}^{}\Phi_{\calK}^* x=\Phi_{\calK}^{}y$}.
The numerical stability of this method depends heavily on the condition number of \smash{$\Phi_{\calK}^{}\Phi_{\calK}^*$} which, in accordance with~\eqref{equation.definition of NERF}, has an upper bound of $\frac{\beta}{\alpha}$.
As such, when designing NERFs, our goal is to make $K\geq M$ as small as possible---meaning we are robust against more erasures---while keeping \smash{$\frac{\beta}{\alpha}$} from becoming too large;
this contrasts with RIP where, as mentioned above, the goal is to make $K\leq M$ as large as possible while keeping $\delta$ small.

The first study of erasure-robust frames was given in~\cite{GoyalKK:01}.
In subsequent years, it was shown that unit norm tight frames are optimally robust against one erasure~\cite{CasazzaK:03} while Grassmannian frames~\cite{StrohmerH:03} are optimally robust against two erasures~\cite{HolmesP:laa04};
this fact is in part responsible for the continued interest in equiangular tight frames~\cite{StrohmerH:03,XiaZG:05,FickusMT:10}.
In a purely algebraic sense, the frames which are most robust against erasures are those that have full spark~\cite{PuschelK:dcc05,AlexeevCM:arxiv11}, namely those collections of vectors $\set{\varphi_n}_{n=1}^{N}$ in $\bbR^M$ that have the property that any $M$ of them form a basis for $\bbR^M$.
However, in the presence of additive noise, what truly matters is conditioning of the subcollection, not its rank.
This realization led to the introduction of NERFs in~\cite{FickusM:12}.
There, in attempting to construct NERFs, the authors ran into many of the same difficulties encountered when attempting to construct RIP matrices.
This is due to the fact that both the RIP~\eqref{equation.definition of RIP} and NERF~\eqref{equation.definition of NERF} conditions are seemingly combinatorial, requiring one to estimate the eigenvalues of $\smash{\binom{N}{K}}$ submatrices of $\Phi$.
Indeed, to date, the best known low-redundancy NERFs are obtained using random matrix theory~\cite{FickusM:12}, essentially because when coupled with a union bound, such constructions allow one to consider each $M\times K$ submatrix of $\Phi$ independently.
Such constructions are becoming increasingly important as NERFs are proving useful outside of the original communications-based applications that inspired them; they can, for example, be used to reconstruct a signal using only the magnitudes of its frame coefficients~\cite{AlexeevBFM:12}.

In the next section, we exploit the subtle differences in definition between NERFs and RIP matrices to introduce a new trick which will enable us to estimate NERF bounds much faster than we can estimate restricted isometry constants.
The key idea is to evaluate the entire frame analysis operator over a large grid of points on the sphere---an $\varepsilon$-net---and, at each point, \textit{sort} the resulting inner products according to magnitude.
This trick reduces the number of computations from being exponential in $N$ to being exponential in $M$.
In Section~$3$, we then show how to exploit ideas from group frames to further reduce the computational complexity from being exponential in $M$ to being subexponential in $M$.
There, the method involves using a large group of orthogonal matrices---the symmetry group of the cube, for example---to simultaneously generate both the NERF and the $\varepsilon$-net used to estimate its bound.
Essentially, this trick allows us to work with an $\varepsilon$-net that only covers a very small portion of the unit sphere, namely the set of all unit norm vectors with nonnegative entries sorted in nondecreasing order.
In the fourth section, we then show how to explicitly construct such an $\varepsilon$-net, and discuss how its size compares with another $\varepsilon$-net generated by a popular method.
Finally, in Section~$5$, we conclude with numerical experimentation, using our theory to construct explicit examples of NERFs and concretely estimate their bounds.


\section{Estimating NERF bounds with $\varepsilon$-nets}

To motivate the results of this section, we begin with a few observations on NERF bounds.
Given $M\leq K\leq N$, recall that if $\set{\varphi}_{n=1}^{N}$ is a $(K,\alpha,\beta)$-NERF for $\bbR^M$ then \eqref{equation.definition of NERF} holds for all $K$-element subsets $\calK\subseteq\set{1,\dotsc,N}$.
To better understand the problem of constructing NERFs, let's find an explicit expression for the \textit{optimal lower and upper NERF bounds}, namely the largest value of $\alpha$ and smallest value of $\beta$ for which~\eqref{equation.definition of NERF} holds for all $\calK$.
To do this, we first note that for any fixed subset $\calK$, the largest value of $\alpha$ and smallest value of $\beta$ for which~\eqref{equation.definition of NERF} holds are
\begin{equation}
\label{equation.derivation of optimal NERF bounds 1}
\alpha_\calK
:=\min_{\norm{x}=1}\sum_{n\in\calK}\abs{\ip{x}{\varphi_n}}^2,
\qquad
\beta_\calK
:=\max_{\norm{x}=1}\sum_{n\in\calK}\abs{\ip{x}{\varphi_n}}^2,
\end{equation}
namely the smallest and largest eigenvalues of $\Phi_{\calK}^{}\Phi_{\calK}^*$, respectively.
As such, the largest value of $\alpha$ and smallest value of $\beta$ for which~\eqref{equation.definition of NERF} simultaneously holds for all $K$-element subsets $\calK$ are
\begin{equation}
\label{equation.definition of optimal NERF bounds}
\alpha_K
:=\min_{\abs{\calK}=K}\alpha_\calK
=\min_{\abs{\calK}=K}\min_{\norm{x}=1}\sum_{n\in\calK}\abs{\ip{x}{\varphi_n}}^2,
\qquad
\beta_K
:=\max_{\abs{\calK}=K}\beta_\calK
=\max_{\abs{\calK}=K}\max_{\norm{x}=1}\sum_{n\in\calK}\abs{\ip{x}{\varphi_n}}^2,
\end{equation}
respectively.
That is, computing the optimal NERF bounds \eqref{equation.definition of optimal NERF bounds} involves finding the extreme eigenvalues of each of the $\smash{\binom{N}{K}}$ subframe operators of the form $\Phi_{\calK}^{}\Phi_{\calK}^*$.
For even modestly-sized choices of $N$ and $K$, this leads to an enormous amount of computation;
as noted in the introduction, this parallels the computational difficulties that have, for several years now, stymied serious progress on the deterministic RIP construction problem.

We now make three key observations which, taken together, give a different perspective on the NERF bound estimation problem.
First, note that we can interchange the order of optimization in~\eqref{equation.definition of optimal NERF bounds}:
\begin{equation}
\label{equation.derivation of optimal NERF bounds 2}
\alpha_K
=\min_{\norm{x}=1}\min_{\abs{\calK}=K}\sum_{n\in\calK}\abs{\ip{x}{\varphi_n}}^2,
\qquad
\beta_K
=\max_{\norm{x}=1}\max_{\abs{\calK}=K}\sum_{n\in\calK}\abs{\ip{x}{\varphi_n}}^2.
\end{equation}
That is, whereas~\eqref{equation.definition of optimal NERF bounds} asks us to ``take any of the \smash{$\binom{N}{K}$} subsets of $\set{1,\dotsc,N}$ and then find the unit norm vectors $x\in\bbR^M$ which are most orthogonal/most parallel to $\set{\varphi_n}_{n\in\calK}$," the alternative, but equivalent formulation~\eqref{equation.derivation of optimal NERF bounds 2} asks us to ``take any unit norm $x\in\bbR^M$, and then find the $K$ vectors in $\set{\varphi_n}_{n=1}^{N}$ which are most orthogonal/most parallel to it."
The second key observation is that for any fixed unit norm $x$, the subsets $\calK\subseteq\set{1,\dotsc,N}$ for which $\sum_{n\in\calK}\abs{\ip{x}{\varphi_n}}^2$ is minimized and maximized are quickly and easily obtained by sorting the values $\bigset{\abs{\ip{x}{\varphi_n}}^2}_{n=1}^{N}$ and then summing the $K$ smallest values and largest values, respectively.
That is,
\begin{equation}
\label{equation.derivation of optimal NERF bounds 3}
\alpha_K
=\min_{\norm{x}=1}\sum_{n=1}^K\abs{\ip{x}{\varphi_{\sigma(n)}}}^2,
\qquad
\beta_K
=\max_{\norm{x}=1}\sum_{n=N-K+1}^N\abs{\ip{x}{\varphi_{\sigma(n)}}}^2,
\end{equation}
where $\sigma$ is an $x$-dependent permutation of $\set{1,\dotsc,N}$ chosen so that the values $\bigset{\abs{\ip{x}{\varphi_{\sigma(n)}}}^2}_{n=1}^{N}$ are arranged in nondecreasing order.
The third and final key observation is that in the envisioned applications of NERFs, we do not need to know $\alpha_K$ and $\beta_K$ exactly; it suffices to have good estimates of them.
In particular, it should not be necessary to evaluate
\begin{equation}
\label{equation.derivation of optimal NERF bounds 4}
\alpha_K(x):=\sum_{n=1}^K\abs{\ip{x}{\varphi_{\sigma(n)}}}^2,
\qquad
\beta_K(x):=\sum_{n=N-K+1}^N\abs{\ip{x}{\varphi_{\sigma(n)}}}^2,
\end{equation}
at each of the infinitely many points $x$ on the unit sphere in $\bbR^M$.
Rather, for the purposes of estimating $\alpha_K$ and $\beta_K$, it should suffice to evaluate~\eqref{equation.derivation of optimal NERF bounds 4} over a large, finite grid of points on the sphere, namely over an $\varepsilon$-net.

To be precise, a set of points $\set{\psi_p}_{p=1}^{P}$ is said to be an \textit{$\varepsilon$-net} for a set $\Omega$ equipped with a metric $d$ if for every $x\in\Omega$ there exists $\psi_p$ such that $d(x,\psi_p)\leq\varepsilon$.
It turns out that it is most convenient for us to work with an $\varepsilon$-net for the projective sphere in $\bbR^M$ under the chordal distance \smash{$d(x_1,x_2):=\sqrt{1-\abs{\ip{x_1}{x_2}}^2}$}.
In particular, given $\varepsilon>0$ and letting $S^{M-1}$ denote the unit sphere in $\bbR^M$, we want
\begin{equation}
\label{equation.definition of epsilon net}
\set{\psi_p}_{p=1}^{P}\subseteq\bbS^{M-1} \text{ s.t. }\forall x\in\bbS^{M-1},\ \exists p\text{ s.t. } \abs{\ip{x}{\psi_p}}^2\geq1-\varepsilon^2.
\end{equation}
Given such a net we, being inspired by~\eqref{equation.derivation of optimal NERF bounds 3}, define \textit{$\varepsilon$-approximate lower} and \textit{upper NERF bounds} by
\begin{equation}
\label{equation.definition of approximate NERF bounds}
\alpha_{K,\varepsilon}
:=\min_{p=1,\dotsc,P}\sum_{n=1}^K\abs{\ip{\psi_p}{\varphi_{\sigma(n)}}}^2,
\qquad
\beta_{K,\varepsilon}
:=\max_{p=1,\dotsc,P}\sum_{n=N-K+1}^N\abs{\ip{\psi_p}{\varphi_{\sigma(n)}}}^2,
\end{equation}
where $\sigma$ is a $p$-dependent permutation of $\set{1,\dotsc,N}$ chosen so that the values $\bigset{\abs{\ip{\psi_p}{\varphi_{\sigma(n)}}}^2}_{n=1}^{N}$ are arranged in nondecreasing order.
That is, given an $\varepsilon$-net~\eqref{equation.definition of epsilon net}, we form the estimates~\eqref{equation.definition of approximate NERF bounds} by the following algorithm: for each $p=1,\dotsc,P$, compute $\Phi^*\psi_p=\set{\ip{\psi_p}{\varphi_n}}_{n=1}^{N}$, then sort $\bigset{\abs{\ip{\psi_p}{\varphi_n}}^2}_{n=1}^{N}$ in nondecreasing order, and let $\alpha_{K,\varepsilon}$ and $\beta_{K,\varepsilon}$ be the sum of the first $K$ and last $K$ resulting values, respectively.
We now make this analysis rigorous, bounding the optimal lower and upper NERF bounds in terms of the estimates~\eqref{equation.definition of approximate NERF bounds}; note that as $\varepsilon$ gets small these bounds become exact.

\begin{theorem}
\label{theorem.NERF bounds for general frames}
Let $M\leq K\leq N$, let $\set{\varphi_n}_{n=1}^{N}\subseteq\bbR^M$, and let \smash{$\set{\psi_p}_{p=1}^{P}$} be any $\varepsilon$-net \eqref{equation.definition of epsilon net} for $\bbS^{M-1}$.
Then, the optimal upper and lower NERF bounds~\eqref{equation.definition of optimal NERF bounds} for $\set{\varphi_n}_{n=1}^{N}$ are estimated by the $\varepsilon$-approximate bounds~\eqref{equation.definition of approximate NERF bounds} according to
\begin{equation}
\label{equation.NERF bounds for general frames 1}
\tfrac1{1-\varepsilon^2}\bigparen{\alpha_{K,\varepsilon}-\tfrac{\varepsilon^2}{1-\varepsilon^2}\beta_{K,\varepsilon}}
\leq\alpha_K
\leq\alpha_{K,\varepsilon},
\qquad
\beta_{K,\varepsilon}
\leq\beta_K
\leq\tfrac1{1-\varepsilon^2}\beta_{K,\varepsilon}.
\end{equation}
\end{theorem}

\begin{proof}
For any $K$-element subset $\calK\subseteq\set{1,\dotsc,N}$, let $\set{\lambda_m}_{m=1}^{M}$ and $\set{u_m}_{m=1}^{M}$ be the eigenvalues and corresponding orthonormal eigenbasis of the $M\times M$ frame operator of the subcollection $\set{\varphi_n}_{n\in\calK}$, yielding the spectral decomposition
\begin{equation*}
\sum_{n\in\calK}\varphi_n^{}\varphi_n^*
=\Phi_{\calK}^{}\Phi_{\calK}^*
=\sum_{m=1}^{M}\lambda_m u_m^{}u_m^*,
\end{equation*}
where $x^*$ denotes the $1\times M$ adjoint (transpose) of some $M\times 1$ vector $x$.
Now, given any $x\in\bbR^M$, conjugating this expression by $x$ gives
\begin{equation}
\label{equation.proof of NERF bounds for general frames 1}
\sum_{n\in\calK}\abs{\ip{x}{\varphi_n}}^2
=x^*\sum_{n\in\calK}\varphi_n^{}\varphi_n^*x
=x^*\Phi_{\calK}^{}\Phi_{\calK}^*x
=x^*\sum_{m=1}^{M}\lambda_m u_m^{}u_m^*x
=\sum_{m=1}^{M}\lambda_m\abs{\ip{x}{u_m}}^2.
\end{equation}
Note that since $\set{u_m}_{m=1}^{M}$ is an orthonormal basis, then for any unit norm $x$, the values \smash{$\bigset{\abs{\ip{x}{u_m}}^2}_{m=1}^{M}$} sum to one, meaning that~\eqref{equation.proof of NERF bounds for general frames 1} corresponds to a weighted average of the eigenvalues $\set{\lambda_m}_{m=1}^{M}$.
Taking these eigenvalues in nondecreasing order, note that \eqref{equation.derivation of optimal NERF bounds 1} gives $\lambda_1=\alpha_{\calK}$ and $\lambda_M=\beta_{\calK}$, meaning these weighted averages lie in the interval $[\alpha_{\calK},\beta_{\calK}]$.
Since the $\varepsilon$-approximate bounds~\eqref{equation.definition of approximate NERF bounds} are examples of such averages for certain choices of $x$ and $\calK$, we immediately obtain two of the four inequalities in~\eqref{equation.NERF bounds for general frames 1}:
\begin{equation*}
\alpha_K
=\min_{\abs{\calK}=K}\alpha_\calK
\leq \alpha_{K,\varepsilon}
\leq \beta_{K,\varepsilon}
\leq \max_{\abs{\calK}=K}\beta_\calK
=\beta_K.
\end{equation*}
Next, to obtain the upper bound on $\beta_K$ given in~\eqref{equation.NERF bounds for general frames 1} we, for any given $\calK$, consider~\eqref{equation.proof of NERF bounds for general frames 1} in the case where $x$ is chosen to be the member of the $\varepsilon$-net $\set{\psi_p}_{p=1}^{P}$ which is guaranteed to be close to the eigenvector $u_M$ corresponding to the largest eigenvalue $\lambda_M=\beta_{\calK}$.
That is, picking $p$ such that \smash{$\abs{\ip{\psi_p}{u_M}}^2\geq 1-\varepsilon^2$} and letting $x=\psi_p$ in~\eqref{equation.proof of NERF bounds for general frames 1} gives
\begin{equation}
\label{equation.proof of NERF bounds for general frames 2}
\sum_{n\in\calK}\abs{\ip{\psi_p}{\varphi_n}}^2
=\sum_{m=1}^{M}\lambda_m\abs{\ip{\psi_p}{u_m}}^2
\geq\lambda_M\abs{\ip{\psi_p}{u_M}}^2
=\beta_{\calK}(1-\varepsilon^2).
\end{equation}
At the same time, the left hand side of~\eqref{equation.proof of NERF bounds for general frames 2} is but one of the $\smash{\binom{N}{K}}$ ways to sum $K$ choices of the values $\bigset{\abs{\ip{\psi_p}{\varphi_n}}^2}_{n=1}^{N}$ and is, of course, no larger than the sum of the $K$ largest values:
\begin{equation}
\label{equation.proof of NERF bounds for general frames 3}
\sum_{n\in\calK}\abs{\ip{\psi_p}{\varphi_n}}^2
\leq \max_{\abs{\calK'}=K}\sum_{n\in\calK'}\abs{\ip{\psi_p}{\varphi_n}}^2
\leq \max_{p'=1,\dotsc,P}\max_{\abs{\calK'}=K}\sum_{n\in\calK'}\abs{\ip{\psi_{p'}}{\varphi_n}}^2
=\beta_{K,\varepsilon}.
\end{equation}
Combining~\eqref{equation.proof of NERF bounds for general frames 2} and~\eqref{equation.proof of NERF bounds for general frames 3}, we see that $\beta_{\calK}(1-\varepsilon^2)\leq\beta_{K,\varepsilon}$ for all $\calK$ and so we obtain the upper bound on $\beta_K$ from~\eqref{equation.NERF bounds for general frames 1}:
\begin{equation}
\label{equation.proof of NERF bounds for general frames 4}
\beta_K
=\max_{\abs{\calK}=K}\beta_\calK
\leq \tfrac1{1-\varepsilon^2}\beta_{K,\varepsilon}.
\end{equation}
We now use a similar approach to find the lower bound on $\alpha_K$; for any given $\calK$, taking $x=\psi_p$ in~\eqref{equation.proof of NERF bounds for general frames 1} gives
\begin{equation*}
\sum_{n\in\calK}\abs{\ip{\psi_p}{\varphi_n}}^2
=\sum_{m=1}^{M}\lambda_m\abs{\ip{\psi_p}{u_m}}^2
\leq\lambda_1\abs{\ip{\psi_p}{u_1}}^2+\lambda_M\sum_{m=2}^{M}\abs{\ip{\psi_p}{u_m}}^2
=\alpha_\calK\abs{\ip{\psi_p}{u_1}}^2+\beta_\calK(1-\abs{\ip{\psi_p}{u_1}}^2).
\end{equation*}
Simplifying, and picking $p$ such that $\psi_p$ is the $\varepsilon$-net member that satisfies $\abs{\ip{\psi_p}{u_1}}^2\geq1-\varepsilon^2$ then yields
\begin{equation}
\label{equation.proof of NERF bounds for general frames 5}
\sum_{n\in\calK}\abs{\ip{\psi_p}{\varphi_n}}^2
=\beta_\calK-(\beta_\calK-\alpha_\calK)\abs{\ip{\psi_p}{u_1}}^2
\leq\beta_\calK-(\beta_\calK-\alpha_\calK)(1-\varepsilon^2)
=(1-\varepsilon^2)\alpha_\calK+\varepsilon^2\beta_\calK.
\end{equation}
At the same time, the left hand side of~\eqref{equation.proof of NERF bounds for general frames 5} is no smaller than the sum of the $K$ smallest values of $\bigset{\abs{\ip{\psi_p}{\varphi_n}}^2}_{n=1}^{N}$, meaning
\begin{equation}
\label{equation.proof of NERF bounds for general frames 6}
\sum_{n\in\calK}\abs{\ip{\psi_p}{\varphi_n}}^2
\geq \min_{\abs{\calK'}=K}\sum_{n\in\calK'}\abs{\ip{\psi_p}{\varphi_n}}^2
\geq \min_{p'=1,\dotsc,P}\min_{\abs{\calK'}=K}\sum_{n\in\calK'}\abs{\ip{\psi_{p'}}{\varphi_n}}^2
=\alpha_{K,\varepsilon}.
\end{equation}
Combining~\eqref{equation.proof of NERF bounds for general frames 5} and~\eqref{equation.proof of NERF bounds for general frames 6} and then using~\eqref{equation.proof of NERF bounds for general frames 4} gives $\alpha_{K,\varepsilon}\leq(1-\varepsilon^2)\alpha_\calK+\varepsilon^2\beta_\calK\leq(1-\varepsilon^2)\alpha_\calK+\tfrac{\varepsilon^2}{1-\varepsilon^2}\beta_{K,\varepsilon}$.
At this point, solving for $\alpha_\calK$ and then minimizing over all $\calK$ gives the lower bound on $\alpha_K$ from~\eqref{equation.NERF bounds for general frames 1}:
\begin{equation*}
\tfrac1{1-\varepsilon^2}\bigparen{\alpha_{K,\varepsilon}-\tfrac{\varepsilon^2}{1-\varepsilon^2}\beta_{K,\varepsilon}}
\leq\min_{\calK}\alpha_\calK
=\alpha_K.\qedhere
\end{equation*}
\end{proof}

We now give some remarks about the previous result and its proof.
Note that the lower bound on $\alpha_K$ in~\eqref{equation.NERF bounds for general frames 1} is more complicated than the upper bound on $\beta_K$, as it involves both the approximate lower and upper frame bounds $\alpha_{K,\varepsilon}$ and $\beta_{K,\varepsilon}$ from~\eqref{equation.definition of approximate NERF bounds}.
For those familiar with frame theory, this is to be expected: lower frame bounds are usually more difficult to estimate than upper frame bounds.
At the same time, this fact is troubling, since the lower estimate is much more important than the upper estimate.
Indeed, for small $K$, there is a real danger that the lower NERF bound could approach zero, see Theorem~11 of \cite{FickusM:12}, for example.
Meanwhile, it is impossible for the upper NERF bound to grow too large: if the entire original frame $\set{\varphi_n}_{n=1}^{N}$ has an upper frame bound of $B$, then the upper NERF bound is at most $B$ since
\begin{equation}
\label{equation.upper estimate in terms of overall frame bound}
\beta_K
=\max_{\norm{x}=1}\max_{\abs{\calK}=K}\sum_{n\in\calK}\abs{\ip{x}{\varphi_n}}^2
\leq\max_{\norm{x}=1}\sum_{n=1}^N\abs{\ip{x}{\varphi_n}}^2
\leq\max_{\norm{x}=1}B\norm{x}^2
=B.
\end{equation}
Moreover, though the estimates~\eqref{equation.definition of approximate NERF bounds} become exact as $\varepsilon$ tends to zero, choosing a small $\varepsilon$ forces the number of elements $P$ in the $\varepsilon$-net~\eqref{equation.definition of epsilon net} to be large, making it prohibitively expensive to compute the approximate NERF bounds~\eqref{equation.definition of approximate NERF bounds}.
In practice, we are therefore compelled to not take $\varepsilon$ to be too small; see the final section for a longer discussion on practical aspects regarding the size of $\varepsilon$.
We are therefore faced with a quandary: to ease our computational burden, we want to take an $\varepsilon$ which is not too small, yet doing so leads to coarse estimates~\eqref{equation.definition of approximate NERF bounds}.
Fortunately, by starting with a good enough frame $\set{\varphi_n}_{n=1}^{N}$ we can find an alternative set of estimates that, according to numerical experimentation, outperform~\eqref{equation.definition of approximate NERF bounds} in cases where $K$ and $\varepsilon$ are not too small.

To be precise, a sequence of vectors $\set{\varphi_n}_{n=1}^{N}$ is said to be \textit{unit norm} if $\norm{\varphi_n}=1$ for all $n$ and is a \textit{tight frame} if $\Phi\Phi^*=A\rmI$ for some $A>0$.
Moreover, if both properties hold true, then $\set{\varphi_n}_{n=1}^{N}$ is called a \textit{unit norm tight frame} (UNTF) and the tight frame constant $A$ is necessarily the \textit{redundancy} $\frac NM$ since $MA=\mathrm{Tr}(A\rmI)=\mathrm{Tr}(\Phi\Phi^*)=\mathrm{Tr}(\Phi^*\Phi)=N$.
In particular, if $\set{\varphi_n}_{n=1}^{N}$ is a UNTF for $\bbR^M$ then $\sum_{n=1}^N\abs{\ip{x}{\varphi_n}}^2=\frac NM\norm{x}^2$ for all $x$ and so~\eqref{equation.upper estimate in terms of overall frame bound} gives $\beta_K\leq\frac MN$.
Armed with this fact, we return to the end of the proof of Theorem~\ref{theorem.NERF bounds for general frames}, again combining~\eqref{equation.proof of NERF bounds for general frames 5} and~\eqref{equation.proof of NERF bounds for general frames 6}.
This time however, we forgo~\eqref{equation.proof of NERF bounds for general frames 4} in favor of the estimate $\beta_K\leq\frac MN$, yielding
$\alpha_{K,\varepsilon}\leq(1-\varepsilon^2)\alpha_\calK+\varepsilon^2\beta_\calK\leq(1-\varepsilon^2)\alpha_\calK+\varepsilon^2\tfrac{N}{M}$ for all $\calK$ and thus
\begin{equation*}
\tfrac1{1-\varepsilon^2}(\alpha_{K,\varepsilon}-\varepsilon^2\tfrac NM)
\leq\min_{\abs{\calK}=K}\alpha_\calK
=\alpha_K.
\end{equation*}
We summarize these facts in the following result.
\begin{theorem}
\label{theorem.NERF bounds for UNTFs}
Let $M\leq K\leq N$, let $\set{\varphi_n}_{n=1}^{N}$ be a unit norm tight frame for $\bbR^M$, and let \smash{$\set{\psi_p}_{p=1}^{P}$} be any $\varepsilon$-net \eqref{equation.definition of epsilon net} for $\bbS^{M-1}$.
Then, the optimal upper and lower NERF bounds~\eqref{equation.definition of optimal NERF bounds} for $\set{\varphi_n}_{n=1}^{N}$ are bounded by $\alpha_{K,\varepsilon}$~\eqref{equation.definition of approximate NERF bounds} according to:
\begin{equation*}
\tfrac1{1-\varepsilon^2}(\alpha_{K,\varepsilon}-\varepsilon^2\tfrac NM)
\leq\alpha_K
\leq\beta_K
\leq\tfrac NM.
\end{equation*}
\end{theorem}
\noindent
To verify the intuition that led up to Theorem~\ref{theorem.NERF bounds for UNTFs}, note that it is indeed stronger than Theorem~\ref{theorem.NERF bounds for general frames} whenever
\begin{equation*}
\tfrac NM
\leq\tfrac1{1-\varepsilon^2}\beta_{K,\varepsilon}
=\tfrac1{1-\varepsilon^2}\max_{p=1,\dotsc,P}\max_{\abs{\calK}=K}\sum_{n\in\calK}\abs{\ip{\psi_p}{\varphi_n}}^2
=\tfrac1{1-\varepsilon^2}\max_{p=1,\dotsc,P}\max_{\abs{\calK}=K}\Bigparen{\tfrac NM-\sum_{n\notin\calK}\abs{\ip{\psi_p}{\varphi_n}}^2}.
\end{equation*}
Simplifying, we see that this is equivalent to having
\begin{equation*}
(1-\varepsilon^2)\tfrac NM
\leq\tfrac NM-\min_{p=1,\dotsc,P}\min_{\abs{\calK}=K}\sum_{n\notin\calK}\abs{\ip{\psi_p}{\varphi_n}}^2.
\end{equation*}
That is, Theorem~\ref{theorem.NERF bounds for UNTFs} outperforms Theorem~\ref{theorem.NERF bounds for general frames} whenever there exists a $p$ and $\calK$ such that $\sum_{n\notin\calK}\abs{\ip{\psi_p}{\varphi_n}}^2\leq\tfrac NM\varepsilon^2$.
Numerical experimentation reveals that this is often the case for the explicit NERF constructions we introduce in the following sections. Indeed, in those examples $\varepsilon$ is not too small and $K$ is a significant fraction of $N$, meaning it is plausible for there to exist at least one $\psi_p$ which is nearly orthogonal to a small number of frame elements $\set{\varphi_n}_{n\notin\calK}$.
Before moving on to those constructions, we conclude this section with a brief discussion of how the above theory does not directly generalize to the deterministic RIP matrix construction problem.

To be clear, $\varepsilon$-nets are already popular subjects in compressed sensing, being fundamental tools of random matrix theory~\cite{RudelsonV:09}.
However, in that literature, $\varepsilon$-nets are exploited analytically rather than computationally, and are seldom constructed explicitly.
In short, fix $K\leq M\leq N$ and consider a tall $M\times K$ submatrix $\Phi_\calK$ of a randomly generated $M\times N$ sensing matrix $\Phi$.
If the rows of $\Phi_\calK$ are chosen randomly, the distribution of $\norm{\Phi_\calK z}^2$ along any given direction $z\in\bbR^K$ is highly concentrated around its mean.
By applying a union bound over an $\varepsilon$-net for $\bbR^K$ that represents the many different choices for $z$, one finds that with high probability, the values of $\norm{\Phi_\calK z}^2$ are a nearly constant function of $z$.
This, in turn, implies that the eigenvalues of $\Phi_\calK^*\Phi_\calK^{}$ are nearly constant, as needed for RIP.

With regards to deterministic constructions of RIP matrices, the problem with this approach is that it essentially makes use of a distinct  $\varepsilon$-net for the column space of $\Phi_\calK$ for each choice of $\calK$.
To be precise, in the NERF construction problem, we have $M\leq K$ and use a single $\varepsilon$-net for $\bbR^M$ to simultaneously estimate the singular values of each \textit{short} $M\times K$ submatrix of $\Phi$.
This stems from the fact that the optimal NERF bounds~\eqref{equation.definition of optimal NERF bounds} can be simplified by changing the order of optimization~\eqref{equation.derivation of optimal NERF bounds 2} and sorting the inner products $\set{\ip{x}{\varphi_n}}_{n=1}^{N}$ according to magnitude~\eqref{equation.derivation of optimal NERF bounds 3},
yielding readily-computable quantities~\eqref{equation.derivation of optimal NERF bounds 4} to be evaluated over this $\varepsilon$-net.
That is, in the NERF problem, the choice of $\varepsilon$-net is independent of $\calK$.
The same argument falls apart when attempting to estimate restricted isometry constants since, as in the random approach, the $\varepsilon$-net must lie in $\bbR^K$, meaning it depends on one's choice of $\calK$.
Formulaically, this is seen by writing the optimal lower restricted isometry constant~\eqref{equation.definition of RIP} in terms of two successive minimizations which cannot be easily interchanged:
\begin{equation*}
\min_{\abs{\calK}=K}\min_{\substack{y\in\bbR^N\\\mathrm{supp}(y)\subseteq\calK}}\biggnorm{\sum_{n=1}^Ny(n)\varphi_n}^2.
\end{equation*}
Of course, this is not to say that the ideas presented here are not applicable to the RIP problem.
Indeed, since the largest eigenvalue of $\Phi_\calK^*\Phi_\calK^{}$ equals that of $\Phi_\calK^{}\Phi_\calK^*$, Theorem~\ref{theorem.NERF bounds for general frames} can be used to estimate the upper restricted isometry constant of a proposed RIP matrix.
However, the smallest eigenvalues of these matrices are essentially unrelated, meaning that a computationally tractable method for estimating lower restricted isometry constants remains elusive.


\section{Group frame constructions of NERFs and $\varepsilon$-nets}

In the previous section, we introduced a new method for numerically estimating optimal NERF bounds~\eqref{equation.definition of optimal NERF bounds}.
This method involves evaluating the frame analysis operator $\Phi^*$ at each point $\psi_p$ of an $\varepsilon$-net~\eqref{equation.definition of epsilon net} and sorting the resulting values \smash{$\bigset{\abs{\ip{\psi_p}{\varphi_n}}^2}_{n=1}^{N}$} in nondecreasing order.
We then use these sorted values to compute the $\varepsilon$-approximate NERF bounds~\eqref{equation.definition of approximate NERF bounds} which estimate the optimal NERF bounds according to Theorems~\ref{theorem.NERF bounds for general frames} and~\ref{theorem.NERF bounds for UNTFs}.
In this section, we discuss how the number of elements in an $\varepsilon$-net grows exponentially with $M$, making this approach computationally infeasible for general frames.
We then show how we can sidestep this computational burden provided the NERF and $\varepsilon$-net are constructed using the action of a large finite group of orthogonal matrices.

Random matrix theorists have found elegant arguments for guaranteeing the existence of efficient $\varepsilon$-nets~\cite{RudelsonV:09}.
For example, given $\varepsilon>0$, we can iteratively construct \smash{$\set{\psi_p}_{p=1}^{P}$} as follows: take any $\psi_1\in\bbS^{M-1}$, and given \smash{$\set{\psi_p}_{p=1}^{k}$}, choose any $\psi_{k+1}\in\bbS^{M-1}$ whose Euclidean distance from each previous $\psi_p$ is greater than $\varepsilon$.
Note that for any $k$, the closed balls around $\set{\psi_p}_{p=1}^{k}$ of radius $\frac{\varepsilon}2$ are disjoint from each other and, by the triangle inequality, lie inside the ball of radius $1+\frac{\varepsilon}2$ centered at the origin.
As such, the total volume of these tiny balls is less than the volume of the large one, meaning $k(\frac{\varepsilon}2)^M\leq(1+\frac{\varepsilon}2)^M$.
In particular, we see that this process must terminate at some $k=P$, where $P\leq(1+\frac2{\varepsilon})^M$.
Moreover, note this process only terminates when the $\varepsilon$-balls around \smash{$\set{\psi_p}_{p=1}^{P}$} cover $\bbS^{M-1}$.
This means \smash{$\set{\psi_p}_{p=1}^{P}$} is an $\varepsilon$-net for $\bbS^{M-1}$, both in terms of Euclidean distance as well as chordal distance~\eqref{equation.definition of epsilon net}, since for any $x\in\bbS^{M-1}$ there exists $p$ such that
\begin{equation}
\label{equation.deriving net for sphere}
1-\abs{\ip{x}{\psi_p}}^2
=(1+\ip{x}{\psi_p})(1-\ip{x}{\psi_p})
\leq2(1-\ip{x}{\psi_p})
=\norm{x-\psi_p}_2^2
\leq\varepsilon^2.
\end{equation}
In summary, this means that for any given $\varepsilon>0$, there exists an $\varepsilon$-net~\eqref{equation.definition of epsilon net} for $\bbS^{M-1}$ of at most \smash{$(1+\frac2{\varepsilon})^M$} elements.

Despite the elegance of this argument, this result is disheartening from a computational perspective, since $(1+\frac2{\varepsilon})^M$ is enormous for even modest choices of $\varepsilon$ and $M$.
Moreover, though it is possible to slightly improve the above estimates, one can also show---borrowing techniques from Section~3 of \cite{FickusM:12}, for example---that the number of points required for any $\varepsilon$-net for $\bbS^{M-1}$ grows exponentially with $M$.
In particular, for an arbitrary frame $\set{\varphi_n}_{n=1}^{N}$, it is infeasible to numerically compute $\Phi^*\psi_p$ for every choice of $p$.
Indeed, it is this fact which motivates the main idea of this section: we construct frames $\set{\varphi_n}_{n=1}^{N}$ and $\varepsilon$-nets $\set{\psi_p}_{p=1}^{P}$ with such high degrees of symmetry that it suffices to only compute $\Phi^*\psi_p$ over a small subset of the $\psi_p$'s.
In particular, we found that such highly symmetric frames and $\varepsilon$-nets can be constructed using the theory of \textit{group frames}~\cite{ValeW:05}.
As a side effect, the frames generated by this approach happen to be UNTFs, allowing us to use the NERF bound estimates of Theorem~\ref{theorem.NERF bounds for UNTFs} in addition to those of Theorem~\ref{theorem.NERF bounds for general frames}.

To be precise, let \smash{$\calU=\set{U_q}_{q=1}^{Q}$} be a finite subgroup of the group $\rmO(M)$ of all real $M\times M$ orthogonal matrices.
We say that $\set{\varphi_n}_{n=1}^{N}$ is \textit{$\calU$-invariant} if
\begin{equation}
\label{equation.definition of group invariance}
\forall q=1,\dotsc,Q,\ \exists\text{ a permutation $\sigma$ of $\set{1,\dotsc,N}$ s.t. }U_q\varphi_n=\pm\varphi_{\sigma(n)}, \forall n=1,\dotsc,N.
\end{equation}
That is, a frame is $\calU$-invariant if each element of the group $\calU$ simply rearranges the frames elements; here, as is common in the frames literature, we make no distinction between a frame element and its negative, as each yields the same outer product, therefore leading to identical frame operators.
As we now discuss, $\calU$-invariant frames permit us to compute their $\varepsilon$-approximate NERF bounds $\alpha_{K,\varepsilon}$ and $\beta_{K,\varepsilon}$~\eqref{equation.definition of approximate NERF bounds} with surprising efficiency, provided the $\varepsilon$-net possesses like symmetry.

Indeed, consider an $\varepsilon$-net obtained by orbiting a finite set of generators under the action of $\calU$.
That is, given $\varepsilon>0$, choose $\set{\psi_r}_{r=1}^{R}\subseteq\bbS^{M-1}$ such that \smash{$\set{U_q\psi_r}_{q=1,}^{Q}\,_{r=1}^{R}$} is an $\varepsilon$-net~\eqref{equation.definition of epsilon net}.
Here, note that computing $\alpha_{K,\varepsilon}$ and $\beta_{K,\varepsilon}$ seemingly involves $QR$ operator-vector multiplications of the form $\Phi^*(U_q\psi_r)$:
for each $q$ and $r$, we first compute $\Phi^*(U_q\varphi_r)=\set{\ip{U_q\psi_r}{\varphi_n}}_{n=1}^{N}$ and then sum the $K$ smallest and largest values \smash{$\bigset{\abs{\ip{U_q\psi_r}{\varphi_n}}^2}_{n=1}^{N}$};
taking the minimum and maximum of these lower and upper sums over all $q$ and $r$ yields $\alpha_{K,\varepsilon}$ and $\beta_{K,\varepsilon}$.
However, only $R$ of these operator-vector multiplications are truly needed: since $\set{\varphi_n}_{n=1}^{N}$ is $\calU$-invariant, then for any $q$ and $r$, rewriting~\eqref{equation.definition of group invariance} as $U_q^*\varphi_n=\pm\varphi_{\sigma^{-1}(n)}$ reveals the values \smash{$\bigset{\abs{\ip{U_q\psi_r}{\varphi_n}}^2}_{n=1}^{N}=\bigset{\abs{\ip{\psi_r}{U_q^*\varphi_n}}^2}_{n=1}^{N}=\bigset{\abs{\ip{\psi_r}{\varphi_{\sigma^{-1}(n)}}}^2}_{n=1}^{N}$} to be a rearrangement of the values \smash{$\bigset{\abs{\ip{\psi_r}{\varphi_n}}^2}_{n=1}^{N}$}.
As such, the sum of the $K$ smallest and largest values of \smash{$\bigset{\abs{\ip{U_q\psi_r}{\varphi_n}}^2}_{n=1}^{N}$} equals the sum of the $K$ smallest and largest values of \smash{$\bigset{\abs{\ip{\psi_r}{\varphi_n}}^2}_{n=1}^{N}$}, respectively.
Thus, we truly only need to evaluate $\Phi^*\psi_r$ for each $r$, a $Q$-fold speedup over the direct method.
We summarize these ideas in the following result.
\begin{theorem}
\label{theorem.computing approximate NERF bounds with group symmetry}
Let $\calU$ be a finite group of orthogonal matrices over $\bbR^M$, and let $\set{\varphi_n}_{n=1}^{N}$ be $\calU$-invariant~\eqref{equation.definition of group invariance}.
Then, choosing any $\set{\psi_r}_{r=1}^{R}\subseteq\bbS^{M-1}$ such that \smash{$\set{U_q\psi_r}_{q=1,}^{Q}\,_{r=1}^{R}$} is an $\varepsilon$-net~\eqref{equation.definition of epsilon net}, the corresponding $\varepsilon$-approximate NERF bounds~\eqref{equation.definition of approximate NERF bounds} can be computed as follows:
\begin{equation*}
\alpha_{K,\varepsilon}
=\min_{r=1,\dotsc,R}\sum_{n=1}^K\abs{\ip{\psi_r}{\varphi_{\sigma(n)}}}^2,
\qquad
\beta_{K,\varepsilon}
=\max_{r=1,\dotsc,R}\sum_{n=N-K+1}^N\abs{\ip{\psi_r}{\varphi_{\sigma(n)}}}^2,
\end{equation*}
where the $r$-dependent permutation $\sigma$ is chosen so that $\bigset{\abs{\ip{\psi_r}{\varphi_{\sigma(n)}}}^2}_{n=1}^{N}$ is arranged in nondecreasing order.
\end{theorem}

In light of this result, we now focus on the problem of constructing explicit $\calU$-invariant frames and $\calU$-generated $\varepsilon$-nets.
In particular, recall from our previous discussion that the number of elements in our $\varepsilon$-net \smash{$\set{U_q\psi_r}_{q=1,}^{Q}\,_{r=1}^{R}$} is necessarily very large.
At the same time, we want to take $R$ as small as possible in order for the approach of Theorem~\ref{theorem.computing approximate NERF bounds with group symmetry} to be computationally tractable.
Together, these facts suggest we should use a group $\calU$ that is very large.

Taking a large $\calU$ however poses a challenge from the perspective of constructing a useful $\calU$-invariant frame $\set{\varphi_n}_{n=1}^{N}$.
Indeed, the most obvious construction of such a frame is as the orbit \smash{$\set{U_q\varphi}_{q=1}^{Q}$} of some $\varphi\in\bbS^{M-1}$ under the action of $\calU$.
In fact, it is known~\cite{ValeW:05} that such frames are necessarily UNTFs---meaning Theorem~\ref{theorem.NERF bounds for UNTFs} applies---provided $\calU$ is \textit{irreducible}, that is, provided the vectors \smash{$\set{U_q x}_{q=1}^Q$} span $\bbR^M$ for any nonzero $x\in\bbR^M$.
Unfortunately, when $\calU$ is very large, such frames are also extremely redundant, making them unattractive from the standpoint of applications.
However, as we see below, this issue can often be addressed by picking $\varphi$ so that \smash{$\set{U_q\varphi}_{q=1}^{Q}$} is actually multiple copies of the same frame $\set{\varphi_n}_{n=1}^{N}$, meaning the actual number of frame elements $N$ will only be a small fraction of $Q$.
Moreover, note that such a frame $\set{\varphi_n}_{n=1}^{N}$ is also a UNTF whenever $\calU$ is irreducible since its frame operator is $\frac NQ$ times the frame operator of the UNTF \smash{$\set{U_q\varphi}_{q=1}^{Q}$}, that is, \smash{$\Phi\Phi^*=\frac NQ\frac QM\rmI=\frac NM\rmI$}.

These facts in hand, we are ready to find the explicit groups, frames and $\varepsilon$-nets needed in Theorem~\ref{theorem.computing approximate NERF bounds with group symmetry}.
We begin by finding large finite irreducible groups of orthogonal matrices.
In this paper we, for the sake of intuitive simplicity, focus on the groups of symmetries of the Platonic solids.
To be clear, there are three types of such solids that exist in $\bbR^M$ for every $M$, namely generalizations of the tetrahedron, octahedron and cube.
The first is the \textit{simplex}, whose symmetry group contains $(M+1)!$ orthogonal matrices, each corresponding to a unique permutation of the simplex's $M+1$ vertices.
An explicit representation of this group is given in~\cite{FickusM:12}, where it is also shown that this group is irreducible.
The second type of Platonic solid in $\bbR^M$ is the \textit{cross-polytope} whose $2M$ vertices are formed by taking the standard basis along with their negatives.
The third type of Platonic solid is the dual of the cross-polytope, namely the \textit{hypercube} whose $2^M$ vertices are formed by taking all $M$-long $\pm1$-valued sequences.

Being duals, the cross-polytope and hypercube have the same symmetry group, namely the set of all $2^M M!$ \textit{signed permutation} matrices obtained by multiplying each of the $M!$ possible permutation matrices by each of the $2^M$ possible $\pm1$-diagonal matrices.
Note this group is irreducible since the set of all signed permutations of any nonzero $x$ necessarily span $\bbR^M$: for any nonzero $x'\in\bbR^M$, picking $q$ and $q'$ so that $U_q x$ and $U_{q'}x'$ are both nonnegative and nondecreasing, we have that $(U_q x)(M),(U_{q'}x')(M)>0$ and so
\begin{equation*}
0
<(U_q x)(M)~(U_{q'}x')(M)
\leq\ip{U_qx}{U_{q'}x'}
=\ip{(U_{q'}^{-1}U_q^{})x}{x'}.
\end{equation*}
For the remainder of this paper, we focus exclusively  on the case where $\calU$ is the group of all signed permutations as opposed to the group of symmetries of the simplex.
We do this for two reasons: signed permutations are easy to understand, and we seek the largest possible group with which to generate our $\varepsilon$-net.
That said, all of the main ideas below carry over to the simplex case.

To reiterate, we generate our $\varepsilon$-net by taking all signed permutations of a given set of generators $\set{\psi_r}_{r=1}^{R}$.
The number $2^M M!$ of such permutations is enormous.
Indeed, for any fixed $\varepsilon>0$, this number eventually grows faster than the number $(1+\frac2{\varepsilon})^M$ of points in the $\varepsilon$-net we discussed earlier.
Following the argument that led to Theorem~\ref{theorem.computing approximate NERF bounds with group symmetry}, this means the sums of the $K$ smallest and largest values of \smash{$\bigset{\abs{\ip{\psi_r}{\varphi_n}}^2}_{n=1}^{N}$} equal those of \smash{$\bigset{\abs{\ip{U_q\psi_r}{\varphi_n}}^2}_{n=1}^{N}$} for any $q=1,\dotsc,Q$, meaning we can compute our $\varepsilon$-approximate NERF bounds $2^M M!$ times faster than before.

Of course, this requires our frame $\set{\varphi_n}_{n=1}^{N}$ to be $\calU$-invariant,
meaning that every signed permutation of any frame element yields another frame element, modulo negation.
If we are not careful, this leads to impractically redundant frames.
Indeed, if any frame element $\varphi_n$ has entries with distinct absolute values, the total number of frame elements is at least $2^{M-1}M!$, leading to a redundancy $\tfrac NM$ of at least $2^{M-1}(M-1)!$.
The only way to avoid this is to generate $\set{\varphi_n}_{n=1}^{N}$ by taking all signed permutations of some $\varphi$ whose entries are mostly zero, and whose nonzero entries assume only a few distinct values.
For example, when $M=4$, though there are $2^M M!=384$ signed permutations of $\varphi=[1\ 1\ 0\ 0]^*$, only $12$ of these lead to frame elements that are distinct modulo negation, namely
\begin{equation}
\label{equation.4,12 Phi}
\Phi
=\frac1{\sqrt{2}}\left[\begin{array}{rrrrrrrrrrrr}1&1&\phantom{-}1&1&\phantom{-}1&1&0&0&0&0&0&0\\1&-1&0&0&0&0&\phantom{-}1&1&\phantom{-}1&1&0&0\\0&0&1&-1&0&0&1&-1&0&0&\phantom{-}1&1\\0&0&0&0&1&-1&0&0&1&-1&1&-1\end{array}\right].
\end{equation}
To be precise, though there are $4!=24$ distinct $4\times 4$ permutation matrices, there are only $\smash{\binom{4}{2}}=6$ distinct permutations of $\varphi=[1\ 1\ 0\ 0]^*$.
Moreover, though there are $2^4=16$ distinct $4\times 4$ diagonal matrices with diagonal entries $\pm1$, for any given fixed permutation of $\varphi=[1\ 1\ 0\ 0]^*$, only $2$ of these lead to frame elements which are distinct modulo negation.
Indeed, for a general $M$, picking $\varphi$ in this manner leads to a frame of \smash{$N=2\binom{M}{2}=M(M-1)$} elements;
though the redundancy $M-1$ of such a frame is still high for some applications, it is nevertheless much more reasonable than the redundancy of $2^{M-1}(M-1)!$ obtained for a general $\varphi\in\bbS^{M-1}$.

To summarize, for any $M$ we have fixed $\calU$ to be the group of all $M\times M$ signed permutation matrices and have discussed how to use that group to construct examples of $\calU$-invariant UNTFs $\set{\varphi_n}_{n=1}^{N}$.
We now want to use the theory of this section and the previous one to explicitly estimate the optimal NERF bounds of $\set{\varphi_n}_{n=1}^{N}$.
In light of Theorem~\ref{theorem.computing approximate NERF bounds with group symmetry}, all that remains to be done is to choose an $\varepsilon>0$ and construct $\set{\psi_r}_{r=1}^{R}\subseteq\bbS^{M-1}$ such that \smash{$\set{U_q\psi_r}_{q=1,}^{Q}\,_{r=1}^{R}$} is an $\varepsilon$-net; the problem of constructing $\set{\psi_r}_{r=1}^{R}\subseteq\bbS^{M-1}$ is the subject of the following section.


\section{Constructing $\varepsilon$-nets for nondecreasing nonnegative vectors}

Letting $\calU$ be the group of all $M\times M$ signed permutation matrices, note that combining the results of Theorems~\ref{theorem.NERF bounds for general frames}, \ref{theorem.NERF bounds for UNTFs} and \ref{theorem.computing approximate NERF bounds with group symmetry} gives a relatively fast method for estimating the optimal NERF bounds~\eqref{equation.definition of optimal NERF bounds} of a $\calU$-invariant UNTF $\set{\varphi_n}_{n=1}^{N}$, provided we are able to construct a relatively small number of vectors $\set{\psi_r}_{r=1}^{R}\subseteq\bbS^{M-1}$ such that \smash{$\set{U_q\psi_r}_{q=1,}^{Q}\,_{r=1}^{R}$} is an $\varepsilon$-net~\eqref{equation.definition of epsilon net} for $\bbS^{M-1}$.
In this section, we discuss how this last problem is equivalent to constructing a small number of vectors $\set{\psi_r}_{r=1}^{R}$ which form an $\varepsilon$-net for the small portion of the unit sphere that consists of all vectors whose entries are nonnegative and are arranged in nondecreasing order.
We further discuss one explicit method for constructing such a collection $\set{\psi_r}_{r=1}^{R}$.

To be precise, we denote the set of all nonnegative, nondecreasing unit-norm vectors as
\begin{equation}
\label{equation.definition of NNUN}
\Snn^{M-1}:=\set{x\in\bbS^{M-1}: \ 0\leq x(1)\leq\cdots\leq x(M)}.
\end{equation}
Note that the orbit $\set{U_q\psi}_{q=1}^Q$ of any given $\psi\in\bbS^{M-1}$ under the action of $\calU$ is invariant under all signed permutations of $\psi$.
As such, when considering $\varepsilon$-nets of the form \smash{$\set{U_q\psi_r}_{q=1,}^{Q}\,_{r=1}^{R}$} we may, without loss of generality, assume that \smash{$\psi_r\in\Snn^{M-1}$} for all $r$.
Moreover, for such $\set{\psi_r}_{r=1}^{R}$ we now show that \smash{$\set{U_q\psi_r}_{q=1,}^{Q}\,_{r=1}^{R}$} being an $\varepsilon$-net for $\bbS^{M-1}$~\eqref{equation.definition of epsilon net} is equivalent to $\set{\psi_r}_{r=1}^{R}$ being an $\varepsilon$-net for $\Snn^{M-1}$, namely to having
\begin{equation}
\label{equation.definition of epsilon net for NNUN}
\set{\psi_r}_{r=1}^{R}\subseteq\Snn^{M-1} \text{ s.t. }\forall x\in\Snn^{M-1},\ \exists r\text{ s.t. } \ip{x}{\psi_r}\geq\sqrt{1-\varepsilon^2}.
\end{equation}

\begin{lemma}
\label{lemma.covering NNUN}
Let $\set{\psi_r}_{r=1}^{R}\subseteq\Snn^{M-1}$~\eqref{equation.definition of NNUN} and let $\calU=\set{U_q}_{q=1}^{Q}$ be the group all signed $M\times M$ permutation matrices.
Then $\set{\psi_r}_{r=1}^{R}$ is an $\varepsilon$-net for $\Snn^{M-1}$~\eqref{equation.definition of epsilon net for NNUN} if and only if $\set{U_q\psi_r}_{q=1,}^{Q}\,_{r=1}^{R}$ is an $\varepsilon$-net for $\bbS^{M-1}$~\eqref{equation.definition of epsilon net}.
\end{lemma}

\begin{proof}
The ``only if" direction is straightforward, since if $\set{\psi_r}_{r=1}^{R}$ satisfies~\eqref{equation.definition of epsilon net for NNUN}, then for any $x\in\bbS^{M-1}$, we can take $q$ such that \smash{$U_qx\in\Snn^{M-1}$}, and pick $\psi_r$ such that \smash{$\abs{\ip{x}{U_q^{-1}\psi_r}}^2=\abs{\ip{U_qx}{\psi_r}}^2\geq1-\varepsilon^2$}.
This means we can pick the ``$\psi_p$" in~\eqref{equation.definition of epsilon net} to be $U_q^{-1}\psi_r$.

For the less obvious ``if" direction, note that if $\set{\psi_r}_{r=1}^{R}\subseteq\Snn^{M-1}$ has the property that \smash{$\set{U_q\psi_r}_{q=1,}^{Q}\,_{r=1}^{R}$} satisfies~\eqref{equation.definition of epsilon net}, then for any \smash{$x\in\Snn^{M-1}\subseteq\bbS^{M-1}$}, there exist $q$ and $r$ such that \smash{$\abs{\ip{x}{U_q\psi_r}}^2\geq1-\varepsilon^2$}.
Furthermore, by replacing the signed permutation with its negative if necessary, we may assume without loss of generality that $\ip{x}{U_q\psi_r}\geq0$, meaning we actually have \smash{$\ip{x}{U_q\psi_r}=\abs{\ip{x}{U_q\psi_r}}\geq\sqrt{1-\varepsilon^2}$}.
We now claim that if both $x$ and $\psi$ lie in \smash{$\Snn^{M-1}$}, then the signed permutation $U\in\calU$ that maximizes $\ip{x}{U\psi}$ is the identity $\calU=\rmI$.
Note that proving this claim gives the result since it implies \smash{$\ip{x}{\psi_r}\geq\ip{x}{U_q\psi_r}\geq\sqrt{1-\varepsilon^2}$}, as needed for~\eqref{equation.definition of epsilon net for NNUN}.

To prove the claim, note that every signed permutation $U$ is of the form $(U\psi)(m)=(-1)^{\tau(m)}\psi(\sigma(m))$, where $\sigma$ is a permutation of $\set{1,\dotsc,M}$ and $\tau$ is a $\set{0,1}$-valued function defined over $\set{1,\dotsc,M}$.
As such, our goal is to find a $\sigma$ and $\tau$ that maximize \smash{$\ip{x}{U\psi}=\sum_{m=1}^{M}x(m)(-1)^{\tau(m)}\psi(\sigma(m))$}.
Of course, since the entries of $x$ and $\psi$ are nonnegative, such sums only grow larger by taking $\tau(m)=0$ for all $m$.
As such, this problem boils down to finding $\sigma$ such that $\sum_{m=1}^{M}x(m)\psi(\sigma(m))$ is maximized.
To do this, note that for a maximizing $\sigma_0$ and any $m_1<m_2$, evaluating our sum both at $\sigma_0$ as well as at its composition with the two-cycle $\tilde{\sigma}$ that interchanges $\sigma_0(m_1)$ and $\sigma_0(m_2)$ gives
\begin{align*}
0
&\leq\sum_{m=1}^{M}x(m)\psi(\sigma_0(m))-\sum_{m=1}^{M}x(m)\psi((\tilde{\sigma}\circ\sigma_0)(m))\\
&=x(m_1)\psi(\sigma_0(m_1))+x(m_2)\psi(\sigma_0(m_2))-x(m_1)\psi(\sigma_0(m_2))-x(m_2)\psi(\sigma_0(m_1))\\
&=\bigbracket{x(m_1)-x(m_2)}\bigbracket{\psi(\sigma_0(m_1))-\psi(\sigma_0(m_2))}.
\end{align*}
Since $x$ is nondecreasing, this implies that either (i) \smash{$x(m_1)=x(m_2)$} or (ii) \smash{$x(m_1)<x(m_2)$} and $\psi(\sigma_0(m_1))\leq\psi(\sigma_0(m_2))$.
That is, the values of $\psi(\sigma_0(m))$ increase whenever the values of $x(m)$ strictly increase.
Moreover, over intervals where $x(m)$ has constant value, we can rearrange the values $\psi(\sigma_0(m))$ so that they increase there as well; modifying $\sigma_0$ in this way preserves the value of $\sum_{m=1}^{M}x(m)\psi(\sigma(m))$, meaning this new permutation $\sigma_1$ is also a maximizer.
Overall, we see that any $\sigma_0$ which maximizes $\sum_{m=1}^{M}x(m)\psi(\sigma(m))$ yields another maximizer $\sigma_1$ that has the additional property that $\psi\circ\sigma_1$ is nondecreasing.
To conclude, note that although there may be several such $\sigma_1$'s, the fact that $\psi$ is itself nondecreasing implies that $\psi=\psi\circ\sigma_1$, meaning that the maximum value of \smash{$\ip{x}{U\psi}=\sum_{m=1}^{M}x(m)(-1)^{\tau(m)}\psi(\sigma(m))$} is indeed $\ip{x}{\psi}=\sum_{m=1}^{M}x(m)\psi(m)$, as claimed.
\end{proof}

Having Lemma~\ref{lemma.covering NNUN}, we see the true computational advantage offered by Theorem~\ref{theorem.computing approximate NERF bounds with group symmetry}: whereas a direct computation of the $\varepsilon$-approximate bounds~\eqref{equation.definition of approximate NERF bounds} requires evaluating $\Phi^*$ at every point of an $\varepsilon$-net for the entire sphere, we instead only need to evaluate $\Phi^*$ at every point of an $\varepsilon$-net~\eqref{equation.definition of epsilon net for NNUN} for the small portion of the sphere that consists of nonnegative nondecreasing vectors~\eqref{equation.definition of NNUN}, provided our frame is invariant under signed permutations.
Indeed, note that since the signed permutations of \smash{$\Snn^{M-1}$} form an essential partition of $\bbS^{M-1}$, we have that the surface area of \smash{$\Snn^{M-1}$} is that of $\bbS^{M-1}$ divided by $2^M M!$.
It is therefore plausible for the number of elements $R$ in an $\varepsilon$-net for \smash{$\Snn^{M-1}$} to be much smaller than the bound  \smash{$(1+\frac2{\varepsilon})^M$} on the number of elements in the $\varepsilon$-net for $\bbS^{M-1}$ that we discussed in the previous section.
At the same time, it is unreasonable to expect $R$ to vanish on the order of \smash{$(1+\frac2{\varepsilon})^M/(2^M M!)$}, since you cannot form an $\varepsilon$-net for \smash{$\Snn^{M-1}$} by simply taking the elements of an $\varepsilon$-net for $\bbS^{M-1}$ that happen to lie inside of \smash{$\Snn^{M-1}$}; doing so ignores the fact that \smash{$\Snn^{M-1}$} is increasingly thin as $M$ grows, meaning that for large $M$, many of the elements that cover parts of \smash{$\Snn^{M-1}$} actually lie outside of it.

To get an upper bound on the number of elements $R$ in a decent $\varepsilon$-net for \smash{$\Snn^{M-1}$}, we now mimic the volumetric argument that produced the \smash{$(1+\frac2{\varepsilon})^M$} bound on the number of points in an $\varepsilon$-net for $\bbS^{M-1}$.
Here, it is more convenient to work with $\infty$-balls (cubes) rather than $2$-balls, since the nonnegativity and monotonicity  conditions that define \smash{$\Snn^{M-1}$} must hold for all indices $m$.

To be precise, given $\varepsilon>0$, we iteratively choose $\psi_r\in\Snn^{M-1}$ so that \smash{$\norm{\psi_r-\psi_{r'}}_\infty>M^{-\frac12}\varepsilon$} for all $r'<r$.
We continue to do so until it is no longer possible, forming $\set{\psi_r}_{r=1}^{R}$ that has the property that for every $x\in\Snn^{M-1}$, there exists $r$ such that \smash{$\norm{x-\psi_r}_\infty\leq M^{-\frac12}\varepsilon$}.
Note that such a $\set{\psi_r}_{r=1}^{R}$ is an $\varepsilon$-net for $\Snn^{M-1}$~\eqref{equation.definition of epsilon net for NNUN} since for any $x\in\Snn^{M-1}$, \eqref{equation.deriving net for sphere} gives $1-\abs{\ip{x}{\psi_r}}^2\leq\norm{x-\psi_r}_2^2\leq M\norm{x-\psi_r}_\infty^2\leq\varepsilon^2$.
All that remains is to bound $R$.
To do this, note that the $\infty$-balls around $\set{\psi_r}_{r=1}^{R}$ of radius \smash{$\frac12 M^{-\frac12}\varepsilon$} are disjoint, meaning their total volume \smash{$R(M^{-\frac12}\varepsilon)^M$} is less than the volume of any set that contains them.
To find such a set, take any $x\in\bbR^M$ for which there exists $r$ such that \smash{$\norm{x-\psi_r}_\infty\leq\frac12 M^{-\frac12}\varepsilon$}.
Since $\psi_r\in\Snn^{M-1}$, this means there exists a nonnegative nondecreasing function each of whose entries $\psi_r(m)$ is within \smash{$\frac12 M^{-\frac12}\varepsilon$} units of $x(m)$.
In particular, defining $\omega\in\bbR^M$, \smash{$\omega(m):=(2m-1)\frac12 M^{-\frac12}\varepsilon$}, we see that $x+\omega$ is nonnegative and nondecreasing.
Moreover, the maximum entry of $x+\omega$ is easily bounded from above:
\begin{equation*}
\norm{x+\omega}_\infty
\leq\norm{x-\psi_r}_\infty+\norm{\psi_r}_\infty+\norm{\omega}_\infty
\leq\tfrac12 M^{-\frac12}\varepsilon+\norm{\psi_r}_2+(2M-1)\tfrac12 M^{-\frac12}\varepsilon
\leq1+M^{\frac12}\varepsilon.
\end{equation*}
Writing $\tilde{x}=x+\omega$, these facts together imply that the $\infty$-balls around $\set{\psi_r}_{r=1}^{R}$ of radius \smash{$\frac12 M^{-\frac12}\varepsilon$} all lie inside the set
\begin{equation}
\label{equation.volumetric derivation of NNUN net 1}
-\omega+\set{\tilde{x}\in\bbR^M: \ 0\leq x(1)\leq\cdots\leq x(M)\leq1+M^{\frac12}\varepsilon},
\end{equation}
whose volume equals the fraction of the volume $[2(1+M^{\frac12}\varepsilon)]^M$ of the cube $\set{x\in\bbR^M:\ \norm{x}_\infty\leq1+M^{\frac12}\varepsilon}$ that corresponds to nonnegative nondecreasing vectors $x$.
As the signed permutations partition of this portion of the cube form an essential partition of the entire cube,
we see that the volume of~\eqref{equation.volumetric derivation of NNUN net 1} is \smash{$[2(1+M^{\frac12}\varepsilon)]^M/(2^M M!)=(1+M^{\frac12}\varepsilon)^M/M!$}.
To summarize, for the $\varepsilon$-net $\set{\psi_r}_{r=1}^{R}$ for $\Snn^{M-1}$ that we have constructed in this manner, we necessarily have that the total volume \smash{$R(M^{-\frac12}\varepsilon)^M$} of the tiny cubes is no more than the volume \smash{$(1+M^{\frac12}\varepsilon)^M/M!$} of \eqref{equation.volumetric derivation of NNUN net 1}.
Solving for $R$ and then using Stirling's approximation \smash{$M!\geq\sqrt{2\pi} M^{M+\frac12}\rme^{-M}$} gives
\begin{equation}
\label{equation.volumetric bound on NNUN net}
R
\leq\frac1{M!}\biggparen{M+\frac{\sqrt{M}}{\varepsilon}}^M
\leq\frac{\rme^M}{\sqrt{2\pi M}}\biggparen{1+\frac1{\sqrt{M}\varepsilon}}^M.
\end{equation}
Note that this bound on the number of points $R$ in an $\varepsilon$-net for $\Snn^{M-1}$ is an improvement over the number of points \smash{$(1+\frac2{\varepsilon})^M$}.
Indeed, the base of the exponent in~\eqref{equation.volumetric bound on NNUN net} can remain constant even if we allow \smash{$\varepsilon=\frac1{\sqrt{M}}$}.
That is, for $M$ large, we can produce an excellent estimate for the NERF bounds of a signed-permutation-invariant frame using far less computation than is needed for a coarse estimate of the NERF bounds of a not-so-symmetric frame.

On the other hand, this method for constructing $\set{\psi_r}_{r=1}^{R}$ leaves much to be desired from the computational perspective expressed in Theorem~\ref{theorem.computing approximate NERF bounds with group symmetry}.
Indeed, note that this method of iteratively ``picking $\psi_r$ such that \smash{$\norm{\psi_r-\psi_{r'}}_\infty>M^{-\frac12}\varepsilon$} for all $r'<r$" is not explicit, yielding no actual values for the entries of $\psi_r$ with which to compute $\alpha_{K,\varepsilon}$ and $\beta_{K,\varepsilon}$.
Moreover, the bound~\eqref{equation.volumetric bound on NNUN net} grows exponentially with $M$; if $R$ truly grows at this rate, then even if we did have $\set{\psi_r}_{r=1}^{R}$ explicitly, it would be computationally intractable to compute $\alpha_{K,\varepsilon}$ and $\beta_{K,\varepsilon}$  via Theorem~\ref{theorem.computing approximate NERF bounds with group symmetry} for all but the smallest values of $M$.

As such, the remainder of this section is focused on the problem of constructing an explicit $\varepsilon$-net $\set{\psi_r}_{r=1}^R$ for $\Snn^{M-1}$ which has the property that $R$ grows subexponentially with $M$.
We stress that at this point in the discussion, it is far from clear that such $\varepsilon$-nets even exist; as far as we know, the only way to demonstrate their existence is to use the explicit construction method we now introduce.

When constructing $\varepsilon$-nets~\eqref{equation.definition of epsilon net for NNUN} for $\Snn^{M-1}$, we found it helpful to view the vectors $x,\psi_r\in\Snn^{M-1}$ as nonnegative, nondecreasing unit-norm functions over a discrete real axis $\set{1,\dotsc,M}$.
This perspective allows us to draw inspiration from real analysis, specifically the theory of integration.
In short, an $\varepsilon$-net~\eqref{equation.definition of epsilon net for NNUN} can be viewed as a fixed set of nonnegative, nondecreasing, unit-norm functions $\set{\psi_r}_{r=1}^R$ which has the property that every other such function $x$ looks a lot like one of them.
In the theory of integration, the standard way to approximate a nonnegative, nondecreasing function $x$ is to use a step function.

To be precise, we will first construct an explicit set of nonnegative, nondecreasing step functions $\set{\hat{\psi}_r}_{r=1}^{R}$ and then normalize them $\psi_r:=\hat{\psi}_r/\norm{\hat{\psi}_r}$ to form our $\varepsilon$-net $\set{\psi_r}_{r=1}^{R}\subseteq\Snn^{M-1}$.
In order to retain control of the number $R$ of these step functions, we will only allow the functions $\hat{\psi}_r$ to attain one of $L$ distinct possible positive values \smash{$\set{b_l}_{l=0}^{L-1}$}.
Here, we fix $b_0=1$ and assume the $b_l$'s are sorted in decreasing order.
That is, in a manner similar to Lebesgue integration, we take horizontal slices of any given nonnegative, nondecreasing, unit-norm $x$, decomposing $\set{1,\dotsc,M}$ into the preimages \smash{$x^{-1}(b_{l+1},b_l]=\set{m : b_{l+1}<x(m)\leq b_l}$}.
For any such $x$, we then compute an approximating step function $\hat{\psi}_x$ by ``rounding up" the values of $x$, that is, by defining
\begin{equation}
\label{equation.definition of general step net}
\hat{\psi}_x(m):=\left\{\begin{array}{crcccl}b_l,&\qquad b_{l+1}&<&x(m)&\leq&b_l,\\b_{L-1},&\qquad0&\leq&x(m)&\leq&b_{L-1}.\end{array}\right.
\end{equation}
Note that since $x$ is nonnegative and has unit norm, taking $b_0=1$ ensures that~\eqref{equation.definition of general step net} indeed defines $\hat{\psi}_x$ at every index $m$.
The remaining quantization levels \smash{$\set{b_l}_{l=1}^{L-1}$} are free for us to choose.
Though we investigated spacing these levels uniformly, that is, letting \smash{$b_l=\frac{L-l}L$} for all $l$, we found that exponential spacing led to better, more elegant results about the approximating properties of our step functions.
As such, for any $\delta\in(0,1)$, we let $b_l=\delta^l$ for all $l=0,\dotsc,L-1$, meaning~\eqref{equation.definition of general step net} becomes
\begin{equation}
\label{equation.definition of exponential step net}
\hat{\psi}_x(m):=\left\{\begin{array}{{crcccl}}\delta^l,&\qquad\delta^{l+1}&<&x(m)&\leq&\delta^{l},\\\delta^{L-1},&\qquad 0&\leq&x(m)&\leq&\delta^{L-1}.\end{array}\right.
\end{equation}
Using such exponentially spaced levels makes it straightforward to estimate the inner product between any given $x\in\Snn^{M-1}$ and its corresponding step function $\psi_x:=\hat{\psi}_x/\norm{\hat{\psi}_x}$:
\begin{equation}
\label{equation.exponential step net calculation 1}
\ip{x}{\psi_x}
=\frac{\ip{x}{\hat{\psi}_x}}{\norm{\hat{\psi}_x}}
=\frac{\displaystyle\sum_{m=1}^{M}x(m)\hat{\psi}_x(m)}{\displaystyle\biggparen{\,\sum_{m=1}^{M}\abs{\hat{\psi}_x(m)}^2}^{\frac12}}.
\end{equation}
To be clear, we want a lower bound on~\eqref{equation.exponential step net calculation 1}, as needed for our $\varepsilon$-net~\eqref{equation.definition of epsilon net for NNUN}.
Moreover, a lower bound on the numerator in~\eqref{equation.exponential step net calculation 1} follows immediately from the fact that we ``round up" in \eqref{equation.definition of exponential step net}, implying $\hat{\psi}_x(m)\geq x(m)$ for all $m$, and thus
\begin{equation}
\label{equation.exponential step net calculation 2}
\sum_{m=1}^{M}x(m)\hat{\psi}_x(m)
\geq\sum_{m=1}^{M}\abs{x(m)}^2
=\norm{x}^2
=1.
\end{equation}
Meanwhile, an upper bound on the denominator in~\eqref{equation.exponential step net calculation 1} can be obtained by exploiting the exponential spacing of our quantization levels.
Indeed, for any index $m$ such that $x(m)>\delta^{L-1}$, we know there exists $l=0,\dotsc,L-2$ such that $\delta^{l+1}< x(m)\leq \delta^{l}$; here, the fact that $x(m)\leq1$ follows from $x(m)\leq\norm{x}=1$.
For such $m$, \eqref{equation.definition of exponential step net} then gives $\hat{\psi}_x(m)=\delta^l=\delta^{-1}\delta^{l+1}<\delta^{-1}x(m)$.
In particular, the square of the denominator in~\eqref{equation.exponential step net calculation 1} satisfies
\begin{align}
\nonumber
\sum_{m=1}^{M}\abs{\hat{\psi}_x(m)}^2
&=\sum_{\substack{m=1\\x(m)>\delta^{L-1}}}^{M}\abs{\hat{\psi}_x(m)}^2+\sum_{\substack{m=1\\x(m)\leq\delta^{L-1}}}^{M}\abs{\hat{\psi}_x(m)}^2\\
\nonumber
&\leq\sum_{\substack{m=1\\x(m)>\delta^{L-1}}}^{M}\delta^{-2}\abs{x(m)}^2+\sum_{\substack{m=1\\x(m)\leq\delta^{L-1}}}^{M}\delta^{2(L-1)}\\
\label{equation.exponential step net calculation 3}
&\leq\delta^{-2}+M\delta^{2(L-1)}.
\end{align}
Putting~\eqref{equation.exponential step net calculation 2} and~\eqref{equation.exponential step net calculation 3} together gives our lower bound on \eqref{equation.exponential step net calculation 1}:
\begin{equation}
\label{equation.exponential step net calculation 4}
\ip{x}{\psi_x}
\geq[\delta^{-2}+M\delta^{2(L-1)}]^{-\frac12}.
\end{equation}
At this point we pick $\delta\in(0,1)$ so that the bound~\eqref{equation.exponential step net calculation 4} is as strong as possible, meaning we pick $\delta$ to be the square root of the $x\in(0,1)$ that minimizes $\frac1x+Mx^{L-1}$.
For $L\geq 2$, calculus reveals this minimizing $x=\delta^2$ to be $[M(L-1)]^{-\frac1L}$, at which point~\eqref{equation.exponential step net calculation 4} becomes
\begin{equation}
\label{equation.exponential step net calculation 5}
\ip{x}{\psi_x}
\geq\bigparen{\tfrac{L-1}L}^\frac12[(L-1)M]^{-\frac1{2L}}.
\end{equation}
As a sanity check, note that for any fixed $M$, the right hand side of~\eqref{equation.exponential step net calculation 5} approaches $1$ from below as $L$ grows large.
This makes sense, since for a large number of quantization levels, we expect our step function $\psi_x$ to be a very good approximation of $x$, meaning $\ip{x}{\psi_x}\approx\ip{x}{x}=1$.
At the same time, for any fixed $L$, this same quantity vanishes as $M$ grows large.
This is due to the fact that using any fixed number of levels results in increasingly poor approximations of some high-dimensional $x$'s.
For an $\varepsilon$-net~\eqref{equation.definition of epsilon net for NNUN}, we seek a compromise between these two extremes.
In particular, for any given $\varepsilon>0$ and $M$, \eqref{equation.exponential step net calculation 5} suggests we take $L$ such that
\smash{$\bigparen{\tfrac{L-1}L}^\frac12[(L-1)M]^{-\frac1{2L}}\geq\sqrt{1-\varepsilon^2}$}.
Rearranging this expression, we summarize these facts in the following result.
\begin{lemma}
\label{lemma.sufficient number of levels}
For any $M$ and $\varepsilon$, take $\delta=[M(L-1)]^{-\frac1{2L}}$ and $L\geq 2$ such that
\begin{equation}
\label{equation.sufficient number of levels}
(L-1)(1-\varepsilon^2)^L\leq\tfrac1M\bigparen{\tfrac{L-1}L}^L.
\end{equation}
Then for any $x\in\Snn^{M-1}$, the step function $\psi_x=\hat{\psi}_x/\norm{\hat{\psi}_x}$, where $\hat{\psi}_x$ is defined in~\eqref{equation.definition of exponential step net}, satisfies $\ip{x}{\psi_x}>\sqrt{1-\varepsilon^2}$.
\end{lemma}

To get a better idea of how large $L$ has to be in order for~\eqref{equation.sufficient number of levels} to hold, note that $\bigparen{\tfrac{L-1}L}^L$ approaches $\frac1\rme$ for large $L$.
More precisely, making use of some easily checked facts, we find that for any $L\geq 2$,
\begin{equation*}
-\log\Bigbracket{\bigparen{\tfrac{L-1}L}^L}
=1+L\int_{\frac{L-1}L}^{1}\!\int_y^1 \frac1{x^2}\,\rmd x\rmd y
\leq1+L\int_{\frac{L-1}L}^{1}\!\int_y^1 \tfrac{L^2}{(L-1)^2}\,\rmd x\rmd y
=1+\tfrac{L}{2(L-1)^2}
\leq 2,
\end{equation*}
and so $\bigparen{\tfrac{L-1}L}^L\geq\frac1{\rme^2}$.
In particular, in order for~\eqref{equation.sufficient number of levels} to hold, it suffices to pick $L\geq 2$ so that
\begin{equation}
\label{equation.sufficient number of levels restated}
(L-1)(1-\varepsilon^2)^L\leq\tfrac1{M\rme^2}.
\end{equation}
In essence, this means that for any fixed $\varepsilon$, we can take $L$ to grow as a logarithm of $M$.
This is significant since the size of $L$ strongly affects the number of distinct step functions of the form $\psi_x$.

To elaborate, note that in light of Lemma~\ref{lemma.sufficient number of levels}, we would like to define our $\varepsilon$-net  $\set{\psi_r}_{r=1}^{R}$ for $\Snn^{M-1}$ as the set of all distinct step functions of the form $\psi_x$ for some $x\in\Snn^{M-1}$.
Unfortunately, it is difficult to describe this set of $\psi_x$'s explicitly.
Instead, we settle for a slightly larger set that is easier to parametrize.
To this end, note that for any nonnegative, nondecreasing, unit-norm $\psi_x$, we have that the pre-normalized step function $\hat{\psi}_x$~\eqref{equation.definition of general step net} is also nondecreasing.
As such, each $\psi_x$ lives in the finite set
\begin{equation}
\label{equation.definition of explicit exponential step net}
\Snn^{M-1}(L,\delta)
:=\Biggset{\psi=\frac{\hat{\psi}}{\norm{\hat{\psi}}}\ \Bigg|\ \hat{\psi}:\set{1,\dotsc,M}\rightarrow\set{\delta^l}_{l=0}^{L-1},\ \, \hat{\psi}(1)\leq\cdots\leq \hat{\psi}(M)},
\end{equation}
and so we let $\set{\psi_r}_{r=1}^{R}$ be an enumeration of the points in \smash{$\Snn^{M-1}(L,\delta)$}.
To find the number of elements $R$ in this $\varepsilon$-net, note that each \smash{$\psi\in\Snn^{M-1}(L,\delta)$} arises from a unique nonincreasing function $\eta:\set{1,\dotsc,M}\rightarrow\set{0,\dotsc,L-1}$ so that $\hat{\psi}(m)=\delta^{\eta(m)}$ for all $m$.
That is, $R$ is the number of all such functions $\eta$.
Since there are at most $L$ choices of $\eta(m)$ for each $m=1,\dotsc,M$, we clearly have the upper bound \smash{$R\leq M^L$}.
To compute $R$ exactly, note that each nonincreasing $\eta$ corresponds to a unique way in which $M$ can be written as a sum of $L$ nonnegative integers, each integer representing the number of times $\eta$ attains a given value $l$.
This is one of the classical ``stars and bars" problems of combinatorics: each choice of $\eta$ corresponds to a unique way of placing $L-1$ ``bars" in the spaces between $M+L$ ``stars," and then removing one star from each of the $L$ resulting blocks.
That is, we have
\begin{equation*}
R
=\binom{M+L-1}{L-1}
\leq M^L.
\end{equation*}
For any fixed $\varepsilon>0$, recalling from~\eqref{equation.sufficient number of levels restated} that $L$ can grow logarithmically with $M$, we therefore see that~\eqref{equation.definition of explicit exponential step net} defines an explicit $\varepsilon$-net for $\Snn^{M-1}$ whose cardinality is a subexponential function of $M$.
This is a dramatic improvement over our volumetric bound~\eqref{equation.volumetric bound on NNUN net} on the size of such $\varepsilon$-nets.
In the next section, we provide numerical experimentation to better indicate exactly how $R$ grows as a function of $M$ and $\varepsilon$.
We conclude this section by combining these final facts with the results of Theorems~\ref{theorem.NERF bounds for general frames}, \ref{theorem.NERF bounds for UNTFs} and~\ref{theorem.computing approximate NERF bounds with group symmetry} as well as Lemmas~\ref{lemma.covering NNUN} and~\ref{lemma.sufficient number of levels}, yielding our main result.
\begin{theorem}
\label{theorem.main result}
Let $\set{\varphi_n}_{n=1}^{N}$ be a unit norm tight frame for $\bbR^M$ which is invariant~\eqref{equation.definition of group invariance} under the action of $M\times M$ signed permutation matrices.
For any $\varepsilon>0$, let $\delta=[M(L-1)]^{-\frac1{2L}}$, take $L\geq 2$ such that $(L-1)(1-\varepsilon^2)^L\leq\tfrac1M\bigparen{\tfrac{L-1}L}^L$, and consider the $\binom{M+L-1}{L-1}$-element collection of step functions  $\set{\psi_r}_{r=1}^{R}=\Snn^{M-1}(L,\delta)$ defined in~\eqref{equation.definition of explicit exponential step net}.

Then for any $M\leq K\leq N$, the optimal NERF bounds~\eqref{equation.definition of optimal NERF bounds} of $\set{\varphi_n}_{n=1}^{N}$ satisfy the estimates
\begin{equation}
\label{equation.main result estimates}
\tfrac1{1-\varepsilon^2}\Bigparen{\alpha_{K,\varepsilon}-\varepsilon^2\min\bigset{\tfrac NM,\tfrac1{1-\varepsilon^2}\beta_{K,\varepsilon}}}
\leq\alpha_K
\leq\alpha_{K,\varepsilon},
\qquad
\beta_{K,\varepsilon}
\leq\beta_K
\leq\min\bigset{\tfrac NM,\tfrac1{1-\varepsilon^2}\beta_{K,\varepsilon}},
\end{equation}
where $\alpha_{K,\varepsilon}$ and $\beta_{K,\varepsilon}$ are found by the following process: for any $r=1,\dotsc,R$, let $\alpha_{K,\varepsilon,r}$ and $\beta_{K,\varepsilon,r}$ be the sums of the $K$ smallest and $K$ largest values of $\bigset{\abs{\ip{\psi_r}{\varphi_n}}^2}_{n=1}^{N}$, respectively; let $\displaystyle\alpha_{K,\varepsilon}=\min_{r}\alpha_{K,\varepsilon,r}$ and $\displaystyle\beta_{K,\varepsilon}=\max_r\beta_{K,\varepsilon,r}$\,.
\end{theorem}


\section{Numerical examples of NERFs generated via symmetry groups}

In the previous section, we provided our main result (Theorem~\ref{theorem.main result}) which encapsulates all the main ideas of the paper, providing a numerical scheme for estimating the optimal NERF bounds of certain highly symmetric frames.
In this section, we offer numerical experimentation to better indicate how useful this result actually is.
In short, we shall see that while the method of Theorem~\ref{theorem.main result} is extremely fast compared to existing methods, it still requires a significant amount of computation for even modest choices of $M$ and $\varepsilon$.

We first describe our approach to implement~Theorem~\ref{theorem.main result}.
Given a fixed $K$, then for each choice of $r=1,\dotsc,R$, we perform the following three steps: we first compute \smash{$\bigset{\abs{\ip{\psi_r}{\varphi_n}}^2}_{n=1}^{N}$} using $\calO(MN)$ operations; we then sort the resulting values in nondecreasing order using $\calO(N\log N)$ operations; we finally sum the $K$ smallest and largest values of \smash{$\bigset{\abs{\ip{\psi_r}{\varphi_n}}^2}_{n=1}^{N}$} to form $\alpha_{K,\varepsilon,r}$ and $\beta_{K,\varepsilon,r}$, respectively.
That is, for each $r=1,\dotsc,R$, we expect to perform $\calO((M+\log N)N)$ operations.
Implementing these calculations as a ``for" loop over all $r=1,\dotsc,R$, we keep track of a running minimum of $\alpha_{K,\varepsilon,r}$ as well as a running maximum of $\beta_{K,\varepsilon,r}$, in the end using $\calO((M+\log N)NR)$ operations to compute $\alpha_{K,\varepsilon}$ and $\beta_{K,\varepsilon}$, as needed for~\eqref{equation.main result estimates}.

Note the above approach presents an opportunity: though Theorem~\ref{theorem.main result} as stated applies to some fixed $K$, the above process can be slightly modified so as to simultaneously estimate the optimal NERF bounds that arise for \textit{every} choice of $K$ between $M$ and $N$.
To be clear, consider the point in the above process where, for any given $r$, we have arranged the values \smash{$\bigset{\abs{\ip{\psi_r}{\varphi_n}}^2}_{n=1}^{N}$} in nondecreasing order.
Rather than fixing $K$ and summing the $K$ first and last values in this list, we can alternatively, at the cost of only an additional $\calO(N)$ operations, compute cumulative sums of this list from both its beginning and end, thereby simultaneously computing the values $\alpha_{K,\varepsilon,r}$ and $\beta_{K,\varepsilon,r}$ for all choices of $K$.
Keeping the running termwise minima and maxima of these cumulative sum functions then produces $\alpha_{K,\varepsilon}$ and $\beta_{K,\varepsilon}$ which, when used in~\eqref{equation.main result estimates}, estimate the optimal NERF bounds of $\set{\varphi_n}_{n=1}^{N}$ for every choice of $K$.
In particular, this trick allows us to determine, at little additional cost, the point at which our lower bound on $\alpha_{K,\varepsilon}$ becomes positive, namely the smallest value of $K$ for which our approach guarantees that every $K$ vectors in $\set{\varphi_n}_{n=1}^{N}$ span $\bbR^M$.

On the topic of speed, we further note that in nearly all the numerical examples we investigated, $R$ turned out to be much larger than both $M$ and $N$.
Since our approach involves $\calO((M+\log N)NR)$ operations overall, this means that the size of $R$ is the key factor in determining whether we can, in a reasonable amount of time, determine the desired estimates for a given frame $\set{\varphi_n}_{n=1}^{N}$ and $\varepsilon>0$.
Moreover, as stated earlier, for any fixed $\varepsilon>0$ the size of the smallest $L\geq 2$ such that \smash{$(L-1)(1-\varepsilon^2)^L\leq\tfrac1M\bigparen{\tfrac{L-1}L}^L$} grows logarithmically with $M$,
and so the number of elements \smash{$R=\binom{M+L-1}{L-1}\leq M^L$} in the $\varepsilon$-net~\eqref{equation.definition of explicit exponential step net} grows subexpontially with $M$.
Though this is an improvement over the exponential growth in $M$ that we expect from~\eqref{equation.volumetric bound on NNUN net}, this number $R$ can still be enormous for even modest choices of $M$ and $\varepsilon>0$.
For example, when applying Theorem~\ref{theorem.main result} to the frame~\eqref{equation.4,12 Phi} of $N=12$ elements in $M=4$ dimensions, the sizes of the smallest possible $L\geq 2$ and resulting value \smash{$\binom{M+L-1}{L-1}$} are, for various values of $\varepsilon$, given in the following table:
\begin{equation}
\label{equation.4,12 R}
\begin{array}{lccc}
\varepsilon^2&L&\binom{M+L-1}{L-1}&R_{\mathrm{improved}}\smallskip\\
\hline\noalign{\smallskip}
2^{-1}          &6      &126        &45         \\
2^{-2}          &19     &7315       &1107       \\
2^{-3}          &47     &230300     &15916      \\
2^{-4}          &110    &6438740    &202628     \\
2^{-5}          &249    &164059875  &2366922
\end{array}
\end{equation}

The values in the fourth column correspond to the improved values of $R$ obtained by working with an $\varepsilon$-net which is a proper subset of the $\varepsilon$-net $\Snn^{M-1}(L,\delta)$ given in~\eqref{equation.definition of explicit exponential step net}.
To be precise, recall that the motivation behind $\Snn^{M-1}(L,\delta)$ is to take an easily-parametrized $\varepsilon$-net which, in accordance with Lemma~\ref{lemma.sufficient number of levels}, contains the normalized versions $\psi_x$ of all the $\set{\delta^l}_{l=0}^{L-1}$-valued step functions $\hat{\psi}_x$ obtained by ``rounding up" any nonnegative, nondecreasing unit norm $x$.
Though the definition~\eqref{equation.definition of explicit exponential step net} of $\Snn^{M-1}(L,\delta)$ relies on the fact that such $x$'s are nonnegative and nondecreasing, it takes no advantage of their unit length.
Indeed, in order for a nondecreasing $\set{\delta^l}_{l=0}^{L-1}$-valued step function $\hat{\psi}$ to arise as the rounded-up version $\hat{\psi}_x$ of some $x\in\Snn^{M-1}$, note it necessarily satisfies:
\begin{equation}
\label{equation.exponential step net condition 1}
1
=\sum_{m=1}^{M}\abs{x(m)}^2
\leq\sum_{m=1}^{M}\abs{\hat{\psi}(m)}^2
=\norm{\hat{\psi}}^2.
\end{equation}
Moreover, recalling the previously used fact that $\hat{\psi}_x(m)<\delta^{-1}x(m)$ for all indices $m$ such that $x(m)>\delta^{L-1}$, such a $\hat{\psi}$ must also satisfy:
\begin{equation}
\label{equation.exponential step net condition 2}
1
=\sum_{m=1}^{M}\abs{x(m)}^2
\geq\sum_{\substack{m=1\\x(m)>\delta^{L-1}}}^{M}\abs{x(m)}^2
\geq\sum_{\substack{m=1\\x(m)>\delta^{L-1}}}^{M}\delta^2\abs{\hat{\psi}(m)}^2
=\delta^2\!\!\!\sum_{\substack{m=1\\\hat{\psi}(m)>\delta^{L-1}}}^{M}\abs{\hat{\psi}(m)}^2.
\end{equation}
In particular, we can obtain an alternative version of Theorem~\ref{theorem.main result} in which the $\varepsilon$-net $\Snn^{M-1}(L,\delta)$ is replaced with a proper subset of itself, namely
\begin{equation}
\label{equation.definition of explicit exponential step net with extra conditions}
\Snn^{M-1}(L,\delta)\cap\Biggset{\psi=\frac{\hat{\psi}}{\norm{\hat{\psi}}}\ :\ \norm{\hat{\psi}}^2\geq1\geq\delta^2\!\!\!\sum_{\substack{m=1\\\hat{\psi}(m)>\delta^{L-1}}}^{M}\abs{\hat{\psi}(m)}^2}.
\end{equation}
In the special case where $M=4$ and $N=2$, the number of elements in~\eqref{equation.definition of explicit exponential step net with extra conditions} for various choices of $\varepsilon$ is given in the last column of~\eqref{equation.4,12 R}.
We note that in practice, we compute the members of this smaller $\varepsilon$-net by still forming each of the \smash{$\binom{M+L-1}{L-1}$} possible choices of $\psi\in\Snn^{M-1}(L,\delta)$.
However, those $\psi$'s that arise from $\hat{\psi}$'s which do not satisfy the additional conditions of~\eqref{equation.definition of explicit exponential step net with extra conditions} are ``skipped" in the process of computing $\alpha_{K,\varepsilon}$ and $\beta_{K,\varepsilon}$.
That is, for such $\psi$, we do not compute, sort, and form cumulative sums of the values \smash{$\bigset{\abs{\ip{\psi}{\varphi_n}}^2}_{n=1}^{N}$}.
For example, implementing this alternative version of Theorem~\ref{theorem.main result} in \textsc{Matlab} and applying it to the $12$-element UNTF~\eqref{equation.4,12 Phi} for $\bbR^4$, we obtain the following values for the $\varepsilon$-approximate lower NERF bound $\alpha_{K,\varepsilon}$ for various values of $\varepsilon$:
\begin{equation}
\label{equation.4,12 alpha_K_epsilon}
\begin{scriptsize}
\begin{array}{lcccccccccccc}
\varepsilon^2   &K=1&K=2&K=3&K=4&K=5&K=6&K=7&K=8&K=9&K={10}&K={11}&K={12}\\
\hline\noalign{\smallskip}
2^{-1}  &0.0000 &0.0000 &0.0000 &0.0000 &0.0000 &0.0000 &0.3821 &0.7275 &1.0039 &1.5811 &2.1068 &3.0000\\
2^{-2}  &0.0000 &0.0000 &0.0000 &0.0000 &0.0000 &0.0000 &0.3824 &0.7193 &1.0003 &1.5213 &2.0325 &3.0000\\
2^{-3}  &0.0000 &0.0000 &0.0000 &0.0000 &0.0000 &0.0000 &0.3821 &0.7192 &1.0000 &1.5085 &2.0117 &3.0000\\
2^{-4}  &0.0000 &0.0000 &0.0000 &0.0000 &0.0000 &0.0000 &0.3820 &0.7192 &1.0000 &1.5036 &2.0047 &3.0000\\
2^{-5}  &0.0000 &0.0000 &0.0000 &0.0000 &0.0000 &0.0000 &0.3820 &0.7192 &1.0000 &1.5015 &2.0021 &3.0000\\
\hline
\alpha_K&0.0000 &0.0000 &0.0000 &0.0000 &0.0000 &0.0000 &0.3820 &0.7192 &1.0000 &1.5000 &2.0000 &3.0000
\end{array}
\end{scriptsize}
\end{equation}
We then use each $\alpha_{K,\varepsilon}$ to find the lower estimate $\tfrac1{1-\varepsilon^2}\Bigparen{\alpha_{K,\varepsilon}-\varepsilon^2\tfrac NM}$ on optimal lower NERF bound $\alpha_K$:
\begin{equation}
\label{equation.4,12 alpha_K_lower bound}
\begin{scriptsize}
\begin{array}{lrrrrrrrrrrrr}
\varepsilon^2   &K=1&K=2&K=3&K=4&K=5&K=6&K=7&K=8&K=9&K={10}&K={11}&K={12}\\
\hline\noalign{\smallskip}
2^{-1}  &-3.0000&-3.0000&-3.0000&-3.0000&-3.0000&-3.0000&-2.2358&-1.5451&-0.9921&0.1621 &1.2135 &3.0000\\
2^{-2}  &-1.0000&-1.0000&-1.0000&-1.0000&-1.0000&-1.0000&-0.4901&-0.0409&0.3337 &1.0284 &1.7100 &3.0000\\
2^{-3}  &-0.4286&-0.4286&-0.4286&-0.4286&-0.4286&-0.4286&0.0081 &0.3934 &0.7143 &1.2955 &1.8705 &3.0000\\
2^{-4}  &-0.2000&-0.2000&-0.2000&-0.2000&-0.2000&-0.2000&0.2075 &0.5672 &0.8667 &1.4038 &1.9383 &3.0000\\
2^{-5}  &-0.0968&-0.0968&-0.0968&-0.0968&-0.0968&-0.0968&0.2975 &0.6457 &0.9355 &1.4532 &1.9699 &3.0000\\
\hline
\alpha_K&0.0000 &0.0000 &0.0000 &0.0000 &0.0000 &0.0000 &0.3820 &0.7192 &1.0000 &1.5000 &2.0000 &3.0000
\end{array}
\end{scriptsize}
\end{equation}
For the purposes of comparison, the last rows of~\eqref{equation.4,12 alpha_K_epsilon} and~\eqref{equation.4,12 alpha_K_lower bound} give the actual numerical values of the optimal lower NERF bound $\alpha_K$ for each $K$.
Note that each $\alpha_K$ lies below its upper estimate~\eqref{equation.4,12 alpha_K_epsilon} and above its lower estimate~\eqref{equation.4,12 alpha_K_lower bound}.
Moreover, as a function of $\varepsilon$, the upper estimate $\alpha_{K,\varepsilon}$ appears to converge to $\alpha_K$ much faster than the lower estimate.
This makes sense, since $\alpha_{K,\varepsilon}$ is the solution to a discretization of the minimization problem~\eqref{equation.definition of optimal NERF bounds} that defines $\alpha_K$.
Regardless, note that even for the not-too-small value $2^{-3}$ of $\varepsilon^2$, we see from the coarse estimates in the corresponding row of~\eqref{equation.4,12 alpha_K_lower bound} that any $7$ of these $12$ vectors form a frame for $\bbR^4$;
looking at the frame itself~\eqref{equation.4,12 Phi}, we see this is the smallest $K$ for which this is true since the last $6$ columns of $\Phi$ clearly do not span $\bbR^4$.
In short, sometimes even not-so-small values of $\varepsilon$ are good enough to produce a positive lower bound on $\alpha_K$.
This is significant since otherwise we may have no information about $\alpha_K$ whatsoever.

To be clear, the ``exact" values of $\alpha_K$ given in the final rows of~\eqref{equation.4,12 alpha_K_epsilon} and~\eqref{equation.4,12 alpha_K_lower bound} were obtained for any $K$ by having \textsc{Matlab} find the minimum of the smallest eigenvalue of \smash{$\Phi_\calK^{}\Phi_\calK^*$} over all $K$-element subsets $\calK\subseteq\set{1,\dotsc,N}$;
this approach is only feasible in this particular example since $N=12$ is so small, meaning \smash{$\binom{N}{K}$} is not too large.
Indeed, for such small values of $N$, this approach is faster than our $\varepsilon$-net-based technique since the number of points in our $\varepsilon$-net~\eqref{equation.4,12 R} becomes enormous when $\varepsilon$ gets small.
For example, on a current generation laptop, we could compute the $\alpha_K$ row of~\eqref{equation.4,12 alpha_K_epsilon} in under a second whereas the $\varepsilon^2=2^{-3}$ and $\varepsilon^2=2^{-5}$ rows took several seconds and over an hour, respectively.
When $M$ and $N$ becomes large, this phenomenon disappears and our methods truly start to shine.

For example, when $M=6$, taking $\set{\varphi_n}_{n=1}^{N}$ to be the \smash{$N=2^{3-1}\binom{6}{3}=80$} signed permutations of $\varphi=[1\ 1\ 1\ 0\ 0\ 0]^*$ which are distinct modulo negation, our \textsc{Matlab} implementation of these ideas took around $8.84$ seconds to find the following lower estimates on $\alpha_K$ for $K=61,\dotsc,80$ using $\varepsilon=\frac12$:
\begin{equation*}
\begin{scriptsize}
\begin{array}{cccccccccccccccccccc}
61&62&63&64&65&66&67&68&69&70&71&72&73&74&75&76&77&78&79&80\\
\hline
0.62&1.22&1.78&2.30&2.89&3.38&3.95&4.46&4.92&5.33&6.06&6.46&7.20&7.93&8.64&9.33&10.25&11.14&12.03&13.33
\end{array}
\end{scriptsize}
\end{equation*}
To do so, it found the smallest sufficient number of quantization levels to be $L=21$, and evaluated the $80\times 6$ analysis operator of $\set{\varphi_n}_{n=1}^{80}$ at each of the $R=32372$ points of the \smash{$\binom{M+L-1}{L-1}=230230$}-element $\varepsilon$-net~\eqref{equation.definition of explicit exponential step net} that satisfy the additional requirements~\eqref{equation.exponential step net condition 1} and~\eqref{equation.exponential step net condition 2}.
Note that in particular, these bounds indicate that any $61$ of the $80$ frame elements span $\bbR^6$.
Obtaining this same fact directly involves forming each of the \smash{$\binom{80}{61}\approx 1.16\times 10^{18}$} such submatrices and showing they have full rank.
More importantly, these bounds indicate how \textit{well} these subcollections span: using the upper bound $\beta_K\leq\frac NM=\frac{80}6$, we see that for any $61$ of these $80$ vectors, the condition number of their frame operator is at most $\frac{80}6\frac1{0.62}\approx21.50$, a very reasonable number for stable reconstruction.

For a more dramatic example, letting $\set{\varphi}_{n=1}^{N}$ be the $N=560$ distinct-modulo-negation signed permutation of $\varphi=[1\ 1\ 1\ 1\ 0\ 0\ 0\ 0]^*$ in $\bbR^8$ and taking $\varepsilon=\frac12$, our \textsc{Matlab} algorithm took about three minutes to show that any $399$ of these vectors spanned $\bbR^8$.
Moreover, it gave the $\varepsilon$-approximate NERF bound $\alpha_{404,\varepsilon}\approx 1.17$ and so we know that the condition number of the frame operator of any $404$ of these vectors is no more than $\frac{560}8\frac1{1.17}\leq60$.
At the same time, it provided similar upper bounds on the condition number of $\Phi_\calK^{}\Phi_\calK^*$ for any $\calK$ such that $\abs{\calK}=K\geq 399$.
To do so, the algorithm took $L=22$ and evaluated the $560\times 8$ analysis operator of $\set{\varphi_n}_{n=1}^{560}$ at the $503487$ points of the $4292145$-element $\varepsilon$-net~\eqref{equation.definition of explicit exponential step net} that satisfy~\eqref{equation.exponential step net condition 1} and~\eqref{equation.exponential step net condition 2}.
While this is indeed a large amount of computation, it pales in comparison to simply forming the $\binom{560}{404}\approx 2.84\times 10^{142}$ possible $8\times 404$ submatrices of $\Phi$, let alone estimating the condition number of each such matrix individually.
Even more dramatically, taking the similarly constructed frame of $N=4032$ elements in $\bbR^{10}$ and letting $\varepsilon=\frac12$, our \textsc{Matlab} algorithm took about one hour and seventeen minutes to show that any $K=2883$ of these elements form a frame for $\bbR^{10}$; the sheer number of such subframes is mind-boggling: $\binom{4032}{2883}\approx 3.65\times 10^{1044}$.

To our knowledge, the theory of this paper is the only currently known method for estimating the optimal NERF bounds for matrices of these sizes.
To be clear, in~\cite{FickusM:12} it is shown that if the entries of $M\times N$ matrix $\Phi$ are independently chosen from a standard normal distribution then $\Phi$ has a high probability of being a good NERF provided $K$ is taken to be no less than about 85\% of $N$.
However, in the example above, we got a meaningful result where $K=399$ is less than $72\%$ of $N$.
To date, the best known explicit construction of a general family of NERFs is to use an equiangular tight frame (ETF) of $M^2-M+1$ vectors in $\mathbb{C}^M$; such frames are robust against the removal of up to half of their frame elements~\cite{FickusM:12}.
Though the examples we present here do not rival the erasure-robustness of these ETF NERFs,
they do provide a lot of design freedom that such frames do not.
In short, the methods used to prove that such ETFs are NERFs do not generalize in the slightest, whereas the methods presented here apply to a whole family of frame constructions.
To be clear, when the ETF-inspired methods are applied in the group frames setting~\cite{FickusM:12}, one finds that $K$ may be taken almost as small as $N-\frac NM$.
However, this fact has a more direct proof:
For any UNTF $\set{\varphi_n}_{n=1}^{N}$ and any $\calK\subseteq\set{1,\dotsc,N}$,
\begin{equation*}
\tfrac NM\norm{x}^2
=\sum_{n=1}^N\abs{\ip{x}{\varphi_n}}^2
=\sum_{n\in\calK}\abs{\ip{x}{\varphi_n}}^2+\sum_{n\notin\calK}\abs{\ip{x}{\varphi_n}}^2
\leq\sum_{n\in\calK}\abs{\ip{x}{\varphi_n}}^2+(N-K)\norm{x}^2,\quad\forall x\in\bbR^M,
\end{equation*}
where the final inequality follows from the Cauchy-Schwarz inequality.  Thus, we see that
\begin{equation*}
[K-(N-\tfrac NM)]\norm{x}^2
\leq\sum_{n\in\calK}\abs{\ip{x}{\varphi_n}}^2
\leq\tfrac NM \norm{x}^2,\quad\forall x\in\bbR^M,
\end{equation*}
meaning that for any $K$ the corresponding optimal NERF bounds satisfy $K-(N-\tfrac NM)\leq\alpha_K\leq\beta_K\leq\tfrac NM$.
In particular, it was already apparent that the previously discussed group frame of $N=560$ vectors in $\bbR^8$ is robust up to $\frac NM=70$ erasures, meaning we could already safely take $K$ as small as $491$.
However, the results of this paper allow us to say much more: applying Theorem~\ref{theorem.main result} in the case where $\varepsilon=\frac12$ tells us $K$ can actually be as small as $399$; applying this same result with a smaller $\varepsilon$ may very well yield even smaller acceptable values of $K$.

\section*{Conclusions and future work}

For many decades, numerical linear algebraists have intensely studied the problem of estimating the extreme eigenvalues of a self-adjoint positive semidefinite matrix.
As a result, we have many excellent algorithms---power methods, QR iterations, etc.---for quickly and accurately computing these eigenvalues for even large matrices.
However, these algorithms cannot be directly applied to many emerging compressed sensing and frame theory problems, since these problems involve estimating the singular values of all submatrices of a given size, leading to a combinatorial nightmare.
What we need are new, computationally tractable algorithms that allow us to estimate the singular values of all of these submatrices simultaneously.

We believe the techniques of this paper are a step in that direction.
However, as seen in the last section, these techniques are still computationally expensive.
Much of this is the fault of our $\varepsilon$-net for $\Snn^{M-1}$ which, despite having far fewer elements than an $\varepsilon$-net for the entire sphere $\bbS^{M-1}$, still contains an enormous number of points for even small choices of $M$.
As such, with regards to taking the results of this paper forward, the most critical problem that needs to be addressed is that of designing more efficient $\varepsilon$-nets for $\Snn^{M-1}$.
Indeed, this seems to be a fundamental problem in analysis: how many nonnegative, nonincreasing unit norm functions do I need to well-approximate every such function, where the quality of the approximation is measured in terms of chordal distance?
Of course, many other important questions arise.
Can the optimal NERF bound estimates of Theorems~\ref{theorem.NERF bounds for general frames} and~\ref{theorem.NERF bounds for UNTFs} be improved?
Is there a better method of performing the quantization that led to Lemma~\ref{lemma.sufficient number of levels}?
Here, we have already tried ``rounding down" as an alternative, but found it to be slightly inferior to the ``rounding up" approach we presented here.
Moreover, how much of the calculation of Theorem~\ref{theorem.main result} can be performed analytically, lessening our computational burden?

\section*{Acknowledgments}
This work was supported by NSF DMS 1042701 and NSF CCF 1017278.
The views expressed in this article are those of the authors and do not reflect the official policy or position of the United States Air Force, Department of Defense, or the U.S.~Government.

\end{document}